\documentclass[letterpaper]{amsart}

\usepackage{amsmath}
\usepackage{amsfonts}
\usepackage{amsthm}
\usepackage{amssymb}

\usepackage{microtype}
\usepackage{verbatim}
\usepackage{enumitem}

\usepackage{mathtools}

\usepackage{graphicx}
\usepackage{caption, subcaption}

\DeclarePairedDelimiterX{\infdivx}[2]{(}{)}{%
  #1\;\delimsize\|\;#2%
}
\newcommand{\infdiv}{\mathrm{KL}\infdivx}

\newcommand{\sC}{\mathcal{C}}
\newcommand{\sE}{\mathcal{E}}

\newcommand{\sL}{\mathcal{L}}

\newcommand{\sG}{\mathcal{G}}
\newcommand{\sT}{\mathcal{T}}
\newcommand{\sS}{\mathcal{S}}

\newcommand{\sF}{\mathcal{F}}

\newcommand{\sM}{\mathcal{M}}

\newcommand{\R}{\mathbb{R}}
\newcommand{\N}{\mathbb{N}}

\DeclareMathOperator{\E}{\mathbf{E}}

\newcommand{\eps}{\varepsilon}

\let\Pr\relax
\DeclareMathOperator{\Pr}{\mathbf{P}}
\newcommand{\V}{\mathbf{Var}}
\DeclareMathOperator{\TV}{TV}
\DeclareMathOperator{\Ber}{Bernoulli}

\DeclareMathOperator{\Bin}{Binomial}
\DeclareMathOperator{\Dirichlet}{Dirichlet}

\DeclareMathOperator{\Beta}{Beta}

\DeclareMathOperator{\Geo}{Geo}

\DeclareMathOperator*{\argmax}{argmax}

\DeclareMathOperator{\UA}{UA}
\DeclareMathOperator{\Aut}{Aut}

\let\originalleft\left
\let\originalright\right
\def\left#1{\mathopen{}\originalleft#1}
\def\right#1{\originalright#1\mathclose{}}

\newcommand{\dif}{\mathop{}\!\mathrm{d}}

\usepackage{xspace}
\newcommand*{\ie}{\emph{i.e.,}\@\xspace}
\newcommand*{\eg}{\emph{e.g.,}\@\xspace}

\makeatletter
\newcommand{\etal}{\emph{et al}\@ifnextchar.{}{.\@}}
\makeatother

\makeatletter
\newcommand{\whp}{\emph{w.h.p}\@ifnextchar.{}{.\@}}
\makeatother

\newcommand{\seclabel}[1]{\label{sec:#1}}

\newcommand{\secref}[1]{\mbox{Section~\ref{sec:#1}}}

\newcommand{\applabel}[1]{\label{app:#1}}

\newcommand{\appref}[1]{\mbox{Appendix~\ref{app:#1}}}

\newcommand{\figlabel}[1]{\label{fig:#1}}

\newcommand{\figref}[1]{\mbox{Figure~\ref{fig:#1}}}

\newcommand{\eqlabel}[1]{\label{eq:#1}}
\renewcommand{\eqref}[1]{(\ref{eq:#1})}

\newtheorem{thm}{Theorem}{\bfseries}{\itshape}
\newcommand{\thmlabel}[1]{\label{thm:#1}}
\newcommand{\thmref}[1]{Theorem~\ref{thm:#1}}
\numberwithin{thm}{section}

\newtheorem{lem}[thm]{Lemma}{\bfseries}{\itshape}
\newcommand{\lemlabel}[1]{\label{lem:#1}}
\newcommand{\lemref}[1]{Lemma~\ref{lem:#1}}

\newtheorem{conj}[thm]{Conjecture}{\bfseries}{\itshape}
\newcommand{\conjlabel}[1]{\label{lem:#1}}
\newcommand{\conjref}[1]{Conjecture~\ref{lem:#1}}

\newtheorem{cor}[thm]{Corollary}{\bfseries}{\itshape}
\newcommand{\corlabel}[1]{\label{cor:#1}}
\newcommand{\corref}[1]{Corollary~\ref{cor:#1}}

\newtheorem{prop}[thm]{Proposition}{\bfseries}{\itshape}
\newcommand{\proplabel}[1]{\label{prop:#1}}
\newcommand{\propref}[1]{Proposition~\ref{prop:#1}}

{\bfseries}{\itshape}

\newtheorem{assumption}[thm]{Assumption}{\bfseries}{\rm}

\theoremstyle{definition}

{\bfseries}{\rm}

\theoremstyle{plain}

\usepackage{hyperref}
\usepackage[dvipsnames]{xcolor}
\definecolor{linkblue}{named}{Blue}
\hypersetup{colorlinks=true, linkcolor=linkblue,  anchorcolor=linkblue,
citecolor=linkblue, filecolor=linkblue, menucolor=linkblue,
urlcolor=linkblue} 

\usepackage{hyphenat}

\usepackage[titletoc]{appendix}

\begin{document}
\title[On Discovery Of The Seed In UA Trees]{On The Discovery Of The Seed \\ In Uniform Attachment Trees}
\date{\today}
\author{Luc Devroye}
\author{Tommy Reddad}
\email{lucdevroye@gmail.com \\
tommy.reddad@gmail.com}
\address{School of Computer Science, McGill University, 3480 University Street, Montr\'{e}al, Qu\'{e}bec, Canada, H3A 2K6}

\subjclass[2010]{Primary: 05C80}
\keywords{Random trees, seed tree, root finding, uniform attachment.}

\maketitle

\begin{abstract}
  We investigate the size of vertex confidence sets for including part
  of (or the entirety of) the seed in seeded uniform attachment trees,
  given knowledge of some of the seed's properties, and with a
  prescribed probability of failure. We also study the problem of
  identifying the leaves of a seed in a seeded uniform attachment
  tree, given knowledge of the positions of all internal nodes of the
  seed.
\end{abstract}

\section{Introduction}
In the web graph, nodes represent websites, and edges represent
links. Its theoretical study requires proper graph models that explain
and approximate what is observed in the field---see \cite{bonato} for
an early book on this topic. One particular set of models grows the
web graph dynamically, one new website at a time. This leads to the
study of random graph dynamics~\cite{durrett, remco}. The main---but
still simple---models here are the uniform random recursive tree and
the preferential attachment model. Simple generalizations of these
tree models in which $k$ instead of one parent are selected at each
step lead to graphs. We are concerned in this paper with the
estimation of the origin of the tree when one is shown the entire
(unrooted) tree. In particular, we consider only uniform attachment
trees that are grown started from a fixed tree $S$, called the seed, a
problem first studied by Bubeck \etal~\cite{seed-influence,
  bubeck-influence1}. Estimating the seed can aid, for example, in the
identification of the source of a rumour, an idea, or an
epidemic. Finding the source of an epidemic can help to identify the
original causes for the spread of the infection, and aid in the
development of preventative measures. The area of discovering the
origins of the web graph or a social network graph is also called
\emph{network archeology}.

Consider the random tree which begins as some fixed tree $S$ and grows
incrementally by adding a new vertex and connecting it to a uniformly
random node among all nodes. We think of the starting tree as the
\emph{seed} of the process, and we continue the attachment process for
a long time. We ask the following question: How difficult is it to
identify the position of the seed in such a process only given the
structure of the tree? In general, what aspects of the seed can be
identified efficiently?

The growth process, started from a single node, yields the uniform
random recursive tree (URRT) or uniform attachment tree (see
\eg~\cite{devroye-records, drmota, moon, na-rapoport}). We take the
following point of view, following \cite{finding-adam}: Given the
uniform attachment tree $T$, and a fixed $\eps > 0$, our algorithm
returns a set $H = H(T, \eps)$ of nodes such that, in the case that
$|S| = 1$,
\[
  \Pr\{V(S) \subseteq H\} \ge 1 - \eps ,
\]
where $V(S)$ denotes the set of vertices of $S$. That this is even
possible regardless of the size of $T$ is interesting. In
\cite{finding-adam}, it is shown that there exist universal constants
$c_1, c_2 > 0$ such that for any algorithm,
\[
  |H| \ge c_1 \exp \left\{ c_2 \sqrt{\log (1/\eps) } \right\} .
\]
Furthermore, an algorithm is given in \cite{finding-adam} that has for
some other universal constants $c_1, c_2 > 0$,
\[
  |H| \le c_1 \exp \left\{ c_2 \frac{ \log (1/\eps) }{\log \log (1/\eps) } \right\} .
\]

Consider now seeds $S$ with $k$ vertices and $\ell$ leaves, where $k$
and $\ell$ are known. We study algorithms that return sets $H$ or
$H^*$, both depending upon $k$, $\ell$, and $\eps$, having the
properties that
\[
  \Pr\{V(S) \subseteq H\} \ge 1 - \eps ,
\]
and
\[
  \Pr\{|V(S) \cap H^*| \ge 1\} \ge 1 - \eps ,
\]
respectively. To be a bit more formal, if we write $L(T)$ for the set
of leaves of the tree $T$, $A^{(m)}$ to be the set of all $m$-sized
subsets of the set $A$, $\sT$ for the space of all unlabelled trees,
and $V(\sT)$ for the set of all vertices in these trees, we consider
algorithms with input $T$ (our tree), $k$, $\ell$, and $\eps$, and
with $K$-sized set-valued output, and introduce the optimal sizes
\begin{align*}
  K(k, \ell, \eps) &= \min\left\{
                     m \colon \begin{array}{l} \displaystyle
                                \exists H_{m, k, \ell, \eps} \colon \sT \to V(\sT)^{(m)} , \text{ such that} \\ \displaystyle
                                \min_{\substack{S \colon |S| = k \\ |L(S)| = \ell}} \Pr\{V(S) \subseteq H_{m, k, \ell, \eps}(T)\} \ge 1 - \eps
                              \end{array}
  \right\} , \\
  K^*(k, \ell, \eps) &= \min\left\{
                       m \colon \begin{array}{l} \displaystyle
                                  \exists H^*_{m, k, \ell, \eps} \colon \sT \to V(\sT)^{(m)} , \text{ such that} \\ \displaystyle
                                  \min_{\substack{S \colon |S| = k \\ |L(S)| = \ell}} \Pr\{|V(S) \cap H^*_{m, k, \ell, \eps}(T)| \ge 1\} \ge 1 - \eps 
                                \end{array}
  \right\} . \\
\end{align*}
In the case that $k = 1$, the seed is a single node, and we simply
write $K(\eps) = K(1, 1, \eps)$.

We can also introduce $K(S, \eps)$ and $K^*(S, \eps)$, the analogous
quantities in which the full structure of $S$ is assumed, where we now
have
\begin{align*}
  K(S, \eps) &= \min\Big\{m \colon \exists H_{m, S, \eps} \colon \sT \to V(\sT)^{(m)} ,\, \Pr\{V(S) \subseteq H_{m, S, \eps}(T)\} \ge 1 - \eps \Big\} , \\
  K^*(S, \eps) &= \min\Big\{m \colon \exists H^*_{m, S, \eps} \colon \sT \to V(\sT)^{(m)} ,\, \Pr\{|V(S) \cap H^*_{m, S, \eps}(T)| \ge 1\} \ge 1 - \eps \Big\} .
\end{align*}

Our main results are as follows:
\begin{thm}\thmlabel{heart-upper}
  There are universal constants $c, \eps_0 > 0$ such that if
  $\eps \le \eps_0$, then
  \[
    K^*(k, \ell, \eps) \le c (1/\eps)^{2/k} \log (1/\eps) .
  \]
\end{thm}
The following result shows that for a fixed $k$, the dependence of
$K^*(k, \ell, \eps)$ is at most subpolynomial in $1/\eps$.
\begin{thm}\thmlabel{heart-upper-subpolynomial}
  There are universal constants $c_1, c_2, c_3, c_4 > 0$ such that for
  $k > c_1$ and $\eps \le \exp\{-c_2  (\log k)^{11}\}$,
  \[
    K^*(k, \ell, \eps) \le c_3 \exp\left\{c_4 \frac{\log (1/\eps)}{\log \log (1/\eps) + \log \log k} \right\} .
  \]
\end{thm}
We note that the bound in \thmref{heart-upper} is better than that of
\thmref{heart-upper-subpolynomial} if $k$ is much larger than
$\log \log (1/\eps)$.

\begin{thm}\thmlabel{heart-lower}
  \[
    K^*(k, \ell, \eps) \ge K((e \eps)^{1/k}) .
  \]
\end{thm}
As an immediate corollary, in view of the lower bound of Bubeck,
Devroye, and Lugosi~\cite[Theorem~4]{finding-adam} on $K(\eps)$, we
have the following explicit lower bound.
\begin{cor}\corlabel{heart-lower}
  There are universal constant $c_1, c_2, c_3 > 0$ such that for all
  $\eps \le e^{-c_1 k}$,
  \[
    K^*(k, \ell, \eps) \ge c_2 \exp\left\{c_3 \sqrt{ \frac{\log (1/\eps)}{k}} \right\} .
  \]
\end{cor}
It should be noted that the above four bounds do not depend on the
value of $\ell$.

\begin{thm}\thmlabel{whole-upper}
  There are universal constants $c, \eps_0 > 0$ such that if
  $\eps \le \eps_0$, then
  \[
    K(k, \ell, \eps) \le (c k \ell/\eps)  \min\Bigl\{\log (k \ell/\eps), K^*(k, \ell, \eps/2)\Bigr\} .
  \]
\end{thm}
In conjunction with \thmref{heart-upper},
\[
  K(k, \ell, \eps) \le (c k \ell/\eps) \min\left\{\log(k \ell/\eps), (1/\eps)^{2/k} \log (1/\eps)\right\} .
\]

We also study the optimal size $K'(k, \ell, \eps)$, which is the size
of the smallest set of vertices to include all leaves of a seed $S$
with $|S| = k$ and $|L(S)| = \ell$, given we already know the position
of its internal nodes. In this case, the dependence of
$K'(k, \ell, \eps)$ is shown to be only logarithmic in $1/\eps$, in
contrast to the preceding results.
\begin{thm}\thmlabel{skeleton-tight}
  There are universal constants $c , \eps_0 > 0$ for which, if
  $\eps \le \eps_0$, then
  \[
    K'(k, \ell, \eps) \le \ell + c (k - \ell) \log \left( (\ell/\eps) \log \left( \frac{k - \ell}{\eps} \right) \right) .
  \]
\end{thm}
\propref{skeleton-tight} shows that \thmref{skeleton-tight} is tight
for a large class of seeds.

We also prove that assuming knowledge of the full structure of the
seed can make things much easier.
\begin{thm}\thmlabel{different-setting}
  There are universal constants $c, \eps_0 > 0$ such that if
  $\eps \le \eps_0$, then
  \[
    K(S_k, \eps) \le c (k + \log(1/\eps)) (1/\eps)^{1/k} \log(1/\eps) ,
  \]
  where $S_k$ is a star on $k$ vertices.
\end{thm}

Some of the above quantities can be related using the following simple
inequalities which we state without proof:
\begin{align}
  K^*(k, \ell, \eps) &\le K(k, \ell, \eps) ; \eqlabel{whole-harder-than-heart} \\
  (k - \ell) + K'(k, \ell, \eps) &\le K(k, \ell, \eps) ; \eqlabel{whole-harder-than-leaves} \\
  K(S, \eps) &\le K(|S|, |L(S)|, \eps) . \eqlabel{whole-harder-than-given}
\end{align}

\subsection{Related work}

The oldest work on this topic seems to be by Haigh~\cite{haigh}, who
in 1969 studied properties of the maximum likelihood estimate of the
root in a uniform attachment tree, including the precise
identification of the limiting probability $1 - \log 2$ of success
when only one candidate node can be selected. Shah and
Zaman~\cite{shah-zaman} studied properties of the maximum likelihood
estimate of the root in a diffusion process over regular
trees~\cite{shah-zaman}. A \emph{diffusion process} over an infinite
graph $G$ is a sequence $(G_1, G_2, \dots)$ described by designating a
root vertex $u_1$, where $G_1 = \{u_1\}$, and $G_{i + 1}$ is obtained
from $G_i$ by adding to $G_i$ a uniformly random edge in its
boundary. Shah and Zaman defined a measure of node centrality called
\emph{rumor centrality} which they showed coincided with the maximum
likelihood estimate for the root of a diffusion process in a regular
tree. In a follow-up work~\cite{shah-zaman-2}, Shah and Zaman study
the effectiveness of rumor centrality as an estimator of the root for
a larger family of random trees. Bubeck, Devroye, and
Lugosi~\cite{finding-adam} independently studied root-finding in
uniform and preferential attachment trees. They showed that rumor
centrality could serve as an effective estimator for the root in a
uniform attachment tree, and gave explicit bounds on the size of
vertex-confidence sets for the root. It is also shown in
\cite{finding-adam} that root-finding is possible in preferential
attachment trees, for instance by picking nodes of high degree. Jog
and Loh~\cite{ling-centrality} showed that root-finding algorithms
also exist for sublinear preferential attachment trees. Khim and
Loh~\cite{khim-confidence} gave bounds on the size of
vertex-confidence sets for the root in a diffusion process over
regular trees. They also showed that root-finding is possible over a
certain asymmetric infinite tree. In general, diffusion processes are
part of the study of first-passage percolation, which is an old and
widely-studied subject. For a recent survey of classical and newer
works in this field, see~\cite{fpp}.

Every root-finding algorithm discussed in the above works depends upon
measures of node centrality. In the uniform and preferential
attachment models, as well as in a diffusion process over $d$-ary
trees, Jog and Loh also showed that with high probability, a single
node persists as the most central node throughout the process, after a
finite number of steps~\cite{jog-persistence}.

Seeded attachment trees have received some more attention lately,
where Bubeck, Mossel, and R\'{a}cz~\cite{bubeck-influence1} first
showed that the total variation distance between between the
distributions of arbitrarily large seeded preferential attachment
trees is lower bounded by a positive constant, as long as the two
seeds have distinct degree profiles. This result was later extended to
hold for all non-isomorphic pairs of seeds by Curien, Duquesne,
Kortchemski, and Manolescu~\cite{curien-influence}, and further
modified to work for uniform attachment trees by Bubeck, Eldan,
Mossel, and R\'{a}cz~\cite{seed-influence}. The influence of the seed
in either sublinear or superlinear preferential attachment trees
remains unknown.

Recently, Lugosi and Pereira~\cite{gabor} specifically studied the
problem of partially recovering the seed in seeded uniform attachment
trees in which the seed is either a path, a star, or a uniform
attachment tree itself. In particular, they show that there are
universal constants $c_1, c_2 > 0$ such that if
$k \ge c_1 \log (1/\eps)$, then
$K(S_k, \eps) \le c_2 k$~\cite[Theorem~3]{gabor}. In comparison to our
\thmref{different-setting}, we note that our result works for all
values of $k$, while their result gives a better joint dependence on
$k$ and $\eps$ in the applicable range.

\subsection{Content of the paper}

To start in \secref{heart}, we focus on algorithms designed to report
sets of vertices which intersect the seed with probability at least
$1 - \eps$. In \secref{heart-upper}, we use basic results about
P\'{o}lya urns to give a simple algorithm which reports a set of nodes
which are known to intersect the seed with probability at least
$1 - \eps$. This gives an upper bound on $K^*(k, \ell, \eps)$. In
\secref{subpolynomial}, we study a different algorithm which gives an
improvement on the dependence of $K^*(k, \ell, \eps)$ on $1/\eps$, for
a certain range of $k$ and $\eps$. We also leverage a lower bound from
\cite{finding-adam} to give a lower bound on $K^*(k, \ell, \eps)$ in
\secref{heart-lower}, which ultimately relies upon the maximum
likelihood estimate for root-finding vertex confidence sets.

In \secref{whole}, our focus shifts to the analysis of algorithms
which return sets which include the entire seed with probability at
least $1 - \eps$. Using a slightly different analysis of the same
algorithm as in \secref{heart-upper}, we give the first upper bound on
$K(k, \ell, \eps)$ from \thmref{whole-upper} in
\secref{whole-upper}. We give yet another simple algorithm to solve
this problem, which relies on the upper bound on $K^*(k, \ell, \eps)$
of \secref{heart-upper}, thereby proving the second part of
\thmref{whole-upper}.

In \secref{leaves}, we investigate algorithms for locating all nodes
of the seed when we assume knowledge of the positions of all internal
nodes. We further show in \secref{whole-star} that if we know that the
seed is a star on $k$ nodes, then only a constant factor of $k$ nodes
suffice to identify all nodes of the seed with constant probability.

\subsection{Notation and background}

For a set $A$ and $i \in \N$, let $A^{(i)}$ denote the set of all
$i$-sized subsets of $A$, \ie
$A^{(i)} = \{ B \colon |B| = i, B \subseteq A\}$.  Let also
$\log \colon (0, \infty) \to \R$ denote the natural logarithm
$\log_e$.

Throughout this document, we write $S$ for a tree with
$V(S) = \{u_1, \dots, u_k\}$ and with $\ell$ leaves. The leaf set of
$S$ is denoted by $L(S)$. In general, we write $|S|$ instead of
$|V(S)|$ for the number of vertices in a tree, so $|S| = k$. The tree
$S$ is called the \emph{seed}, and we assume throughout that
$k \ge 2$.

For a tree $T$, we say that $T'$ is a \emph{subtree} of $T$ if the
vertices of $T'$ induce a connected subgraph of $T$. For $T$ with
subtree $T'$ and $u \in V(T)$, let $(T, T')_{u \downarrow}$ and
$(T, V(T'))_{u \downarrow}$ denote the subtree of $T$ rooted at $u$
``facing away'' from $T'$, \ie $(T, T')_{u \downarrow}$ is the subtree
of $T$ induced on all nodes whose (unique) path to $T'$ includes the
vertex $u$.

For a tree $T$ and set $X \subseteq V(T)$, we write $N_T(X)$ for the set
of neighbours of all vertices in $X$, \ie
\[
  N_T(X) = \{u \in V(T) \setminus X \colon u \text{ is adjacent to $v$ for some $v \in X$} \} .
\]
With a slight abuse of notation, we will write $\deg_T(X) = |N_T(X)|$,
and for $u \in V(T)$, $\deg_T(u) = \deg_T(\{u\})$. We will omit the
subscript in $\deg_T$ and $N_T$ whenever the tree $T$ is understood.

\begin{figure}
  \centering
  \begin{subfigure}[b]{0.45\textwidth}
    \centering
    \includegraphics[width=\textwidth]{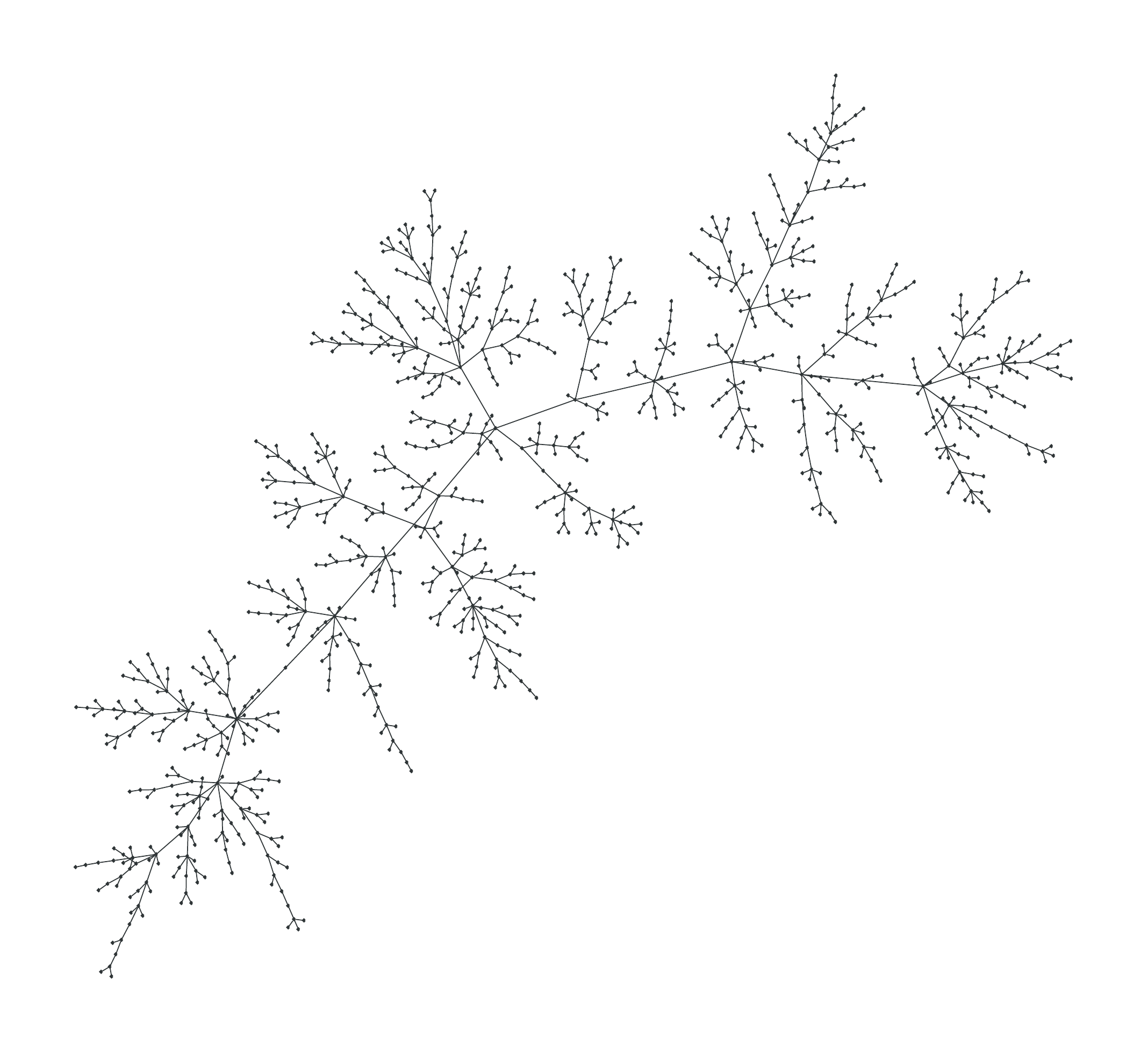}
    \caption{$S$ is a path on $10$ nodes.}
  \end{subfigure}
  \begin{subfigure}[b]{0.45\textwidth}
    \centering
    \includegraphics[width=\textwidth]{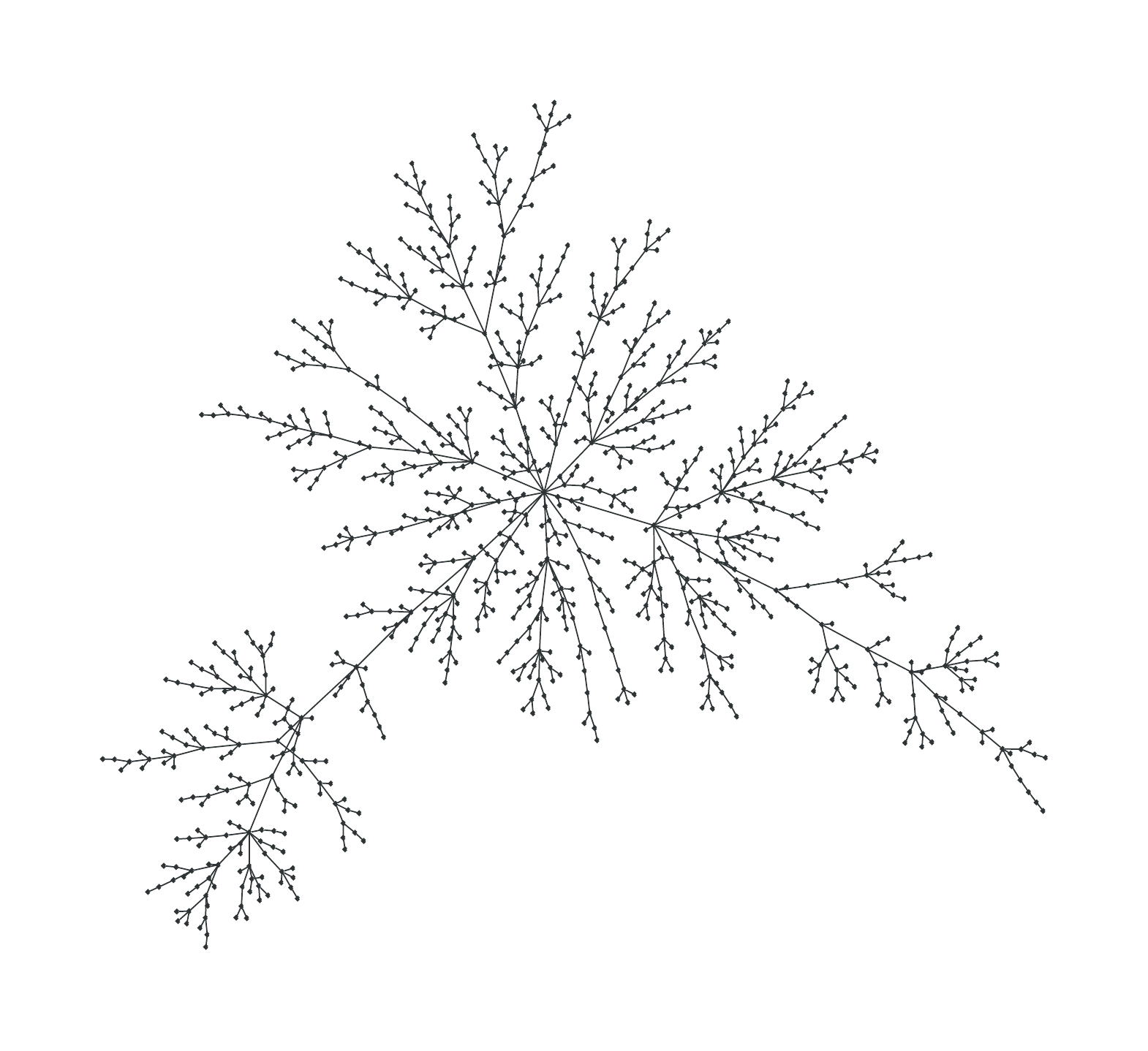}
    \caption{$S$ is a star on $10$ nodes.}
  \end{subfigure}
  \caption{One sample from $\UA(1000, S)$ for two different seeds.}
  \figlabel{ua-samples}
\end{figure}

We inductively define a distribution on labelled trees: Let
$\alpha \ge 0$, let $\UA_\alpha(k, S) = S$, and suppose that we are
given $T_{n - 1} \sim \UA_\alpha(n - 1, S)$, where
$V(T_{n - 1}) = \{u_1, \dots, u_{n - 1}\}$. Let
$T_n \sim \UA_\alpha(n, S)$ be obtained from $T_{n - 1}$ by adding a
leaf labelled $u_n$ to $T_{n - 1}$ and connecting it to a vertex
$u \in \{u_1, \dots, u_{n - 1}\}$ with probability proportional to
$\deg_{T_{n - 1}}(u)^\alpha$. The distribution $\UA_0(n, P_2)$, in
which new leaves are attached to uniformly random vertices
sequentially, is called the \emph{uniform attachment tree} or
\emph{random recursive tree}, and is sometimes denoted by
$\text{UA}(n)$ or $\text{URRT}(n)$ in the literature. The distribution
$\UA_1(n, P_2)$ is called the \emph{(linear) preferential attachment
  tree} or \emph{Barab\'{a}si-Albert model}~\cite{barabasi}, sometimes
denoted by $\text{PA}(n)$ in the literature. For $\alpha \in (0, 1)$,
$\UA_\alpha(n, P_2)$ is called the \emph{sublinear preferential
  attachment tree}, and when $\alpha > 1$, it is called the
\emph{superlinear preferential attachment tree}. For general seeds
$S$, the distributions above are said to be \emph{with seed $S$}. In
this paper, we are mostly concerned with $\UA_0(n, S)$, so we
generally omit the subscript ``0.''  Moreover, if $T \sim \UA(n, S)$
and the distribution of $T$ is understood, we avoid repeating its
distribution. See \figref{ua-samples} for a typical sample from
$\UA(n, S)$ for different choices of $S$.

The vertices of the aforementioned trees are labelled; for a (rooted
or unrooted) labelled tree $T$, let $T^\circ$ denote the isomorphism
class of $T$, \ie the operation $\circ$ ``forgets'' the labelling of
$T$. With some abuse of notation, we refer to nodes of $T^\circ$ using
their original labels in $T$. Of course, if $T \sim \UA(n, S)$ and
nothing at all is assumed about $S$, it is always possible that
$S = T$. Therefore, except if otherwise specified, we assume knowledge
of the size of the seed in trying to locate it in $T^\circ$---our main
goal is to detect $S$ in the unlabelled tree $\UA(n, S)^\circ$, given
that $|S| = k$.

In general, sets of vertices which are made to include parts of the
seed are referred to as \emph{(root-finding) vertex confidence
  sets}. We reserve the letter $H$ for functions which map unlabelled
trees to vertex confidence sets, and such functions are called
\emph{root-finding algorithms}. We also reserve the letter $K$ for the
size of vertex confidence sets.

We assume that the reader is familiar with the basics of probability
theory, including basic properties of standard random variables.  We
write $\Beta(\alpha, \beta)$ for a Beta distribution with parameters
$\alpha, \beta$, \ie the distribution supported on $[0, 1]$ with
density
\[
  f_{\alpha, \beta}(x) = \frac{1}{\mathrm{B}(\alpha, \beta)} x^{\alpha - 1} (1 - x)^{\beta - 1} ,
\]
where for $(\alpha_1, \dots, \alpha_k) \in \R^k_+$, we have
$\mathrm{B}(\alpha_1, \dots, \alpha_k) = \frac{\prod_{i = 1}^k
  \Gamma(\alpha_i)}{\Gamma(\sum_{i = 1}^k \alpha_i)}$. We write
$\Dirichlet(\alpha_1, \dots, \alpha_k)$ for the Dirichlet
distribution, which is supported on the $k$-simplex
\[
  \left\{(x_1, \dots, x_k) \colon \sum_{i = 1}^k x_i = 1 \text{ and } x_i \in [0, 1] \right\} 
\]
and which has density
\[
  f_{\alpha_1, \dots, \alpha_k}(x_1, \dots, x_k) = \frac{1}{\mathrm{B}(\alpha_1, \dots, \alpha_k)} \prod_{i = 1}^k x_i^{\alpha_i - 1} .
\]
We often make use of the following fact about the distribution of sums
of Dirichlet marginals: If
$(X_1, \dots, X_k) \sim \Dirichlet(\alpha_1, \dots, \alpha_k)$, and
$I \subseteq \{1, \dots, k\}$ is some index set, then
\[
  \sum_{i \in I} X_i \sim \Beta\left(\sum_{i \in I} \alpha_i, \sum_{i \in \{1, \dots, k\} \setminus I} \alpha_i\right) .
\]

\subsection{Acknowledgements}
The authors would like to thank the anonymous referees for their
helpful comments and suggestions. Luc Devroye is supported by NSERC
Grant A3456. Tommy Reddad is supported by an NSERC PGS D scholarship
396164433.

\section{Finding a single vertex in the seed}\seclabel{heart}
  \subsection{An upper bound}\seclabel{heart-upper}
  For a tree $T$ and $u \in V(T)$, let 
\[
  \psi_T(u) = \max_{v \in V(T) - u} |(T, u)_{v \downarrow}| ,
\]
so $\psi_T(u)$ is the size of the largest subtree of $T$ hanging off
of the vertex $u$. We omit the subscript $T$ when the tree is
understood.

Consider the strategy for picking central nodes from an unlabelled
copy of $T \sim \UA(n, S)$ which works by simply picking the $K$ nodes
minimizing the values of $\psi$. We will write
$H^*_{\psi; K}(T^\circ)$ for such a set of nodes. This strategy was
first successfully introduced in \cite{finding-adam}, where $\psi$ is
seen as a measure of node centrality, and in which it is shown that
the root of $\UA(n)$ (which is the node labelled $u_1$) is likely to
have low $\psi$ value. In fact, it is shown that this is also true for
the root in preferential attachment trees. This centrality measure is
further studied for root-finding in different settings
\cite{ling-centrality,jog-persistence,khim-confidence}.

Specifically, the result of \cite[Theorem 3]{finding-adam} has that
for $K \ge (2.5/\eps) \log(1/\eps)$,
\[
  \Pr_{T \sim \UA(n)}\Big\{u_1 \in H^*_{\psi; K}(T^\circ)\Big\} \ge 1 - \frac{4 \eps}{1 - \eps} ,
\]
and so, for $\eps \le 1/2$, $K(\eps) \le (80/\eps) \log(1/\eps)$. The
following result indicates that $H^*_{\psi; K}$ also serves as a
root-finding algorithm for seeded uniform attachment seeds, and offers
an upper bound on the size of such vertex confidence sets.

Our proof relies on a few supporting lemmas and a classical result in
the study of P\'{o}lya urns, whose proofs and statements are given in
\appref{app}.
\begin{prop}\proplabel{heart-upper}
  There are universal constants $c, \eps_0 > 0$ for which, if
  $\eps \le \eps_0$ and
  \[
    K \ge c (1/\eps)^{2/k} \log(1/\eps) ,
  \]
  then
  \[
    \Pr\{|V(S) \cap H^*_{\psi; K}(T^\circ)| \ge 1 \} \ge 1 - \eps .
  \]
\end{prop}
\thmref{heart-upper} follows immediately.
\begin{proof}[Proof of \propref{heart-upper}]
  \begin{figure}
    \centering
    \includegraphics[width=0.6\textwidth]{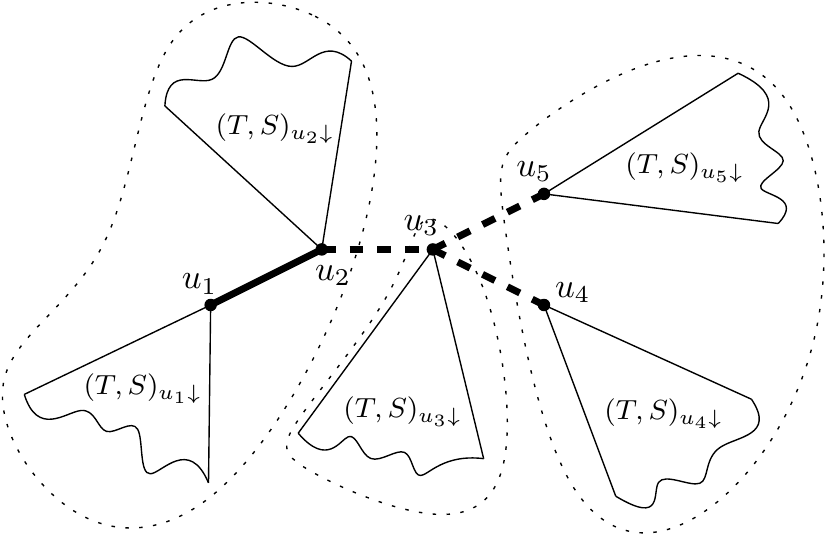}
    \caption{The seed $S$ has vertex set $\{u_1, \dots, u_5\}$. The
      node $r = u_3$ is its unique centroid. The dashed edges form the
      set $\partial_S(r)$. The dotted lines outline the components of
      $S - \partial_S(r)$.}
    \figlabel{centroid}
  \end{figure}
  Let $\psi_* = \min\{\psi(u_1), \dots, \psi(u_k)\}$. If for all
  $i > K$, $\psi(u_i) > \psi_*$, then $V(S)$ intersects
  $H^*_{\psi; K}(T^\circ)$, so
  \begin{align}
    \Pr\{V(S) \cap H^*_{\psi; K}(T^\circ) = \emptyset\} &\le \Pr\{\exists i > K \colon \psi(u_i) \le \psi_*\} \notag \\
                                                        &\le \Pr\{\psi_* \ge nt\} + \Pr\{\exists i > K \colon \psi(u_i) \le nt\} , \eqlabel{phi-tail}
  \end{align}
  where $t > 0$ is to be specified later. It is a classical result
  that the seed $S$ has a centroid, \ie a node $r$ whose removal
  splits the seed into components each of size at most
  $k/2$~\cite{jordan-centroid}. Note that
  \[
    \psi_* \le \psi(r) \le \max_{C \in \sC(S - \partial_S(r))} \sum_{u \in
      C} |(T, S)_{u \downarrow}| ,
  \]
  where $\sC(G)$ denotes the set of components of a graph $G$, and
  $\partial_S(r)$ denotes all of the edges of $S$ incident to
  $r$. Now, for any $C \in \sC(S - \partial_S(r))$, we have by
  \lemref{dir-convergence} that
  \[
    \frac{1}{n} \sum_{u \in C} |(T, S)_{u \downarrow}| \overset{d}{\to} \Beta(|C|, k - |C|) ,
  \]
  as $n \to \infty$, so stochastically,
  \[
    \frac{\psi_*}{n} \le \max_{C \in \sC(S - \partial_S(r))} \Beta(|C|, k - |C|)
  \]
  and since $|C| \le k/2$ uniformly because $r$ is a centroid, then
  each such beta random variable is stochastically smaller than a
  $\Beta(k/2, k/2)$ by \lemref{beta-dom-1}. See \figref{centroid}. Let
  $f_{k/2, k/2}(x)$ be the density of a $\Beta(k/2, k/2)$ random
  variable. Using the bound
  \[
    \sqrt{\frac{2 \pi}{x}} \left(\frac{x}{e}\right)^x \le \Gamma(x) \le \sqrt{\frac{2 \pi}{x}} \left(\frac{x}{e} \right)^x e^{\frac{1}{12 x}} ,
  \]
  which holds for all $x > 0$~\cite[Equation 5.6.1]{nist}, we see that
  \[
    \mathrm{B}(k/2, k/2) = \frac{\Gamma(k/2)^2}{\Gamma(k)} \ge 2^{-k} \cdot \frac{2}{e^{1/12}} \sqrt{\frac{2 \pi}{k}} > 3^{-k} 
  \]
  for all $k \ge 1$. Then,
  \begin{align*}
    \Pr\{\psi_* \ge nt\} &\le k \Pr\{\Beta(k/2, k/2) \ge t\} \\
                                    &= k \int_{t}^1 f_{k/2, k/2}(x) \dif x \\
                              &\le k 3^k \int_0^{1 - t} (x (1 - x))^{k/2 - 1} \dif x \\
                              &\le k 3^k \int_0^{1 - t} x^{k/2 - 1} \dif x \\
                          &= 2 (1 - t)^{k/2} 3^k,
  \end{align*}
  so we can pick $t = 1 - (1/9) (\eps/4)^{2/k}$.

  To summarize, we have shown that we can bound the first term in
  \eqref{phi-tail} by
  \[
    \Pr\{\psi_*/n \ge 1 - (1/9) (\eps/4) ^{2/k}\} \le \eps/2 .
  \]
  For the second term, the argument is identical to that of
  \cite[Theorem 3]{finding-adam}. Indeed, if $T_i$ denotes the
  subgraph of $T$ containing the vertex $u_i$ after the removal of all
  edges between $\{u_1, \dots, u_K\}$, then for any $i > K$,
  \[
    \psi(u_i) \ge \min_{1 \le j \le K} \sum_{m = 1, m \neq j}^K |T_m| ,
  \]
  and by \lemref{dir-convergence},
  \[
    \frac{1}{n} \sum_{m = 1, m \neq j}^K |T_m| \overset{d}{\to} \Beta(K - 1, 1) 
  \]
  as $n \to \infty$, so that
  \begin{align*}
    \Pr\{\exists i > K \colon \psi(u_i) \le nt\} &\le K \Pr\{ \Beta(K - 1, 1) \le t\} \\
                                                                     &= K t^{K - 1} \\
                                                                     &\le K e^{- (1/9) \eps^{2/k} (K - 1)} ,
  \end{align*}
  and a little bit of arithmetic shows that we should pick
  \[
    K \ge c (1/\eps)^{2/k} \log (1/\eps)
  \]
  for universal constants $c, \eps_0 > 0$, as long as
  $\eps \le \eps_0$.
\end{proof}

Consequently, we see that
\[
  K^*(k, \ell, \eps) \le c (1/\eps)^{2/k} \log (1/\eps) .
\]

We note that $H^*_{\psi; K}(T^\circ)$ can be computed in $O(n)$ time.

  \subsection{The maximum likelihood estimate}\seclabel{mle}
  Given an unlabelled tree $T$, and a candidate seed $S$, define the
likelihood function $\sL_{T}(S)$ to be the probability of observing
$T$ under $\UA(n, S)^\circ$, \ie
\[
  \sL_{T}(S) = \Pr_{T' \sim \UA(n, S)}\{T'^\circ = T\} ,
\]
and if $\sS_{k, \ell}(T)$ denotes the set of all possible seeds in $T$
with $k$ vertices and $\ell$ leaves, the \emph{maximum likelihood
  estimate} for $S$ is given by
\[
  S^* = \argmax_{S \in \sS_{k, \ell}(T)} \sL_{T}(S) .
\]
Note that the subtrees $(T, S)_{u \downarrow}$ for $u \in V(S)$ are,
conditionally on their sizes, independent random recursive trees:
\begin{lem}\lemlabel{conditional-indep}
  Let $S$ be some seed, and $T \sim \UA(n, S)$, and $n_u \in \N$ for
  $u \in S$ be such that $\sum_{u \in S} n_u = n$. Then, conditionally
  on $|(T, S)_{u \downarrow}| = n_u$ for all $u \in S$, the trees
  $(T, S)_{u \downarrow}$ are independently distributed as $\UA(n_u)$.
\end{lem}
\begin{proof}[Proof sketch]
  Any incoming node in the attachment process $\UA(n, S)$,
  conditionally upon connecting to $(T, S)_{u \downarrow}$ for some
  $u \in S$, will connect to a uniformly random node of
  $(T, S)_{u \downarrow}$. Conditioning on the event that
  $|(T, S)_{u \downarrow}| = n_u$, this precisely describes the
  uniform attachment tree $\UA(n_u)$.
\end{proof}

As a consequence,
\begin{equation}
  \sL_T(S) = \prod_{u \in V(S)} \sL_{(T, S)_{u \downarrow}}(u) , \eqlabel{mle}
\end{equation}
where $\sL_{(T, S)_{u \downarrow}}(u)$ is the likelihood of the tree
$(T, S)_{u \downarrow}$ to be rooted at $u$, which is computed in
\cite[Section 3]{finding-adam}. Specifically, as in
\cite{finding-adam}, let $T$ be a rooted tree and $v$ a vertex of $T$,
and $T_1, \dots, T_k$ be the subtrees rooted at the children of $v$
listed in an arbitrary order, and $S_1, \dots S_L$ be the isomorphism
classes realized by these subtrees, define
\[
  \Aut(v, T) = \prod_{i = 1}^L \bigl|\bigl\{j \in \{1, \dots, k\} \colon T_j^\circ = S_i\bigr\}\bigr| ! .
\]
Define also, for an unrooted unlabelled tree $T$,
\[
  \overline{\Aut}(u, T) = |\{v \in V(T) \colon (T, v)^\circ = (T, u)^\circ\} | ,
\]
\ie $\overline{\Aut}(u, T)$ is the number of vertices $v \in V(T)$
such that the rooted trees $(T, v)$ and $(T, u)$ are isomorphic. Then,
\[
  \sL_T(u) = \frac{|T|}{\overline{\Aut}(u, T)} \prod_{v \in V(T)} \frac{1}{|(T, u)_{v \downarrow}| \Aut(v, (T, u))} ,
\]
so
\begin{equation}
  \sL_T(S) = \prod_{u \in S} \frac{|(T, S)_{u \downarrow}|}{\overline{\Aut}(u, (T, S)_{u \downarrow})} \prod_{v \in (T, S)_{u \downarrow}} \frac{1}{|(T, S)_{v \downarrow}| \Aut(v, (T, S)_{u \downarrow})} . \eqlabel{likelihood}
\end{equation}
In particular, this implies that the maximum likelihood estimate $S^*$
can be computed in polynomial time in $n$ for any fixed $k, \ell$,
just as in the case when $k = 1$~\cite{finding-adam}.

  \subsection{A vertex-confidence set with size subpolynomial in $(1/\eps)$} \seclabel{subpolynomial}
  For a tree $T$ and $u \in V(T)$, let
\[
  \varphi_T(u) = \prod_{v \in V(T) - \{u\}} |(T, u)_{v \downarrow}| ,
\]
omitting the $T$ subscript when $T$ is understood. We note, as noted
by Bubeck, Devroye, and Lugosi in \cite{finding-adam}, that
$1/\varphi$ resembles the expression of the likelihood in
\eqref{likelihood}, so that nodes with small values of $\varphi$
should be likely candidates for the root. For $K \ge 1$, let
$H^*_{\varphi; K}(T^\circ)$ denote the set of $K$ vertices in
$T^\circ$ minimizing their values of $\varphi$. In
\cite{finding-adam}, it is shown that $H^*_{\varphi; K}$ serves as an
effective root-finding algorithm for $T \sim \UA(n)$. We show that
this is also the case when $T \sim \UA(n, S)$.

\begin{prop}
  There are universal constants $c_1, c_2, c_3, c_4 > 0$ such that if
  $k > c_1$,
  $\eps \le \exp\left\{- c_2 (\log k)^{11} \right\}$, and
  \[
    K \ge c_3 \exp \left\{ c_4 \frac{\log (1/\eps)}{\log \log (1/\eps) + \log \log k } \right\} ,
  \]
  then
  \[
    \Pr\{ |V(S) \cap H^*_{\varphi; K}(T^\circ)| \ge 1\} \ge 1 - \eps .
  \]
\end{prop}
\begin{proof}
  Let $u$ be a vertex of $S$, and label each node
  $v \in (T, S)_{u \downarrow}$ using an extended version of the
  labelling scheme from \cite{finding-adam}, where every node of $T$
  gets a label in
  \[
    \N^* = \N \cup \N^2 \cup \N^3 \cup \dots ,
  \]
  such that if we write $v = (u, j_1, \dots, j_\ell)$, we mean that
  $v \in (T, S)_{u \downarrow}$, and that $v$ is the $j_\ell$-th child
  of $(u, j_1, \dots, j_{\ell - 1})$. Let also
  \[
    s(v) = \sum_{i = 1}^\ell (\ell - i + 1) j_i ,
  \]
  which we denote by $s$ when the node $v$ is understood. Let $u_*$ be
  the node of $S$ minimizing its value of $\varphi$ in $T$, so
  \[
    \varphi(u_*) = \min_{u \in S} \varphi(u) .
  \]
  Let $K$ and $K'$ be related such that $K = |\{v \in (T, S)_{u_* \downarrow} \colon s(v) \le 3 K'\}|$. Then,
  \begin{align}
    \Pr\{u_* \not\in H_{\varphi; K}(T^\circ)\} &\le \Pr\{\exists v \colon s(v) > 3K' \text{ and } \varphi(v) \le \varphi(u_*) \} \notag \\
                                             &= \Pr\{\exists u \in S, \, v \in (T, S)_{u \downarrow} \colon s(v) > 3K' \text{ and } \varphi(v) \le \varphi(u_*)\} \notag \\
                                               &\le \Pr\{\exists u \in S, \, v \in (T, S)_{u \downarrow} \colon s(v) > 3K' \text{ and } \varphi(v) \le \varphi(u)\} \notag \\
                                             &\le \sum_{u \in V(S)} \Pr\{\exists v \in (T, S)_{u \downarrow} \colon s(v) > 3K' \text{ and } \varphi(v) \le \varphi(u)\} ,
  \end{align}
  By the arguments of \cite[Page 9, Equation (9)]{finding-adam}, we have
  \begin{align}
    \Pr\{\exists v \in (T, S)_{u \downarrow} \colon s(v) > 3K' \text{ and } \varphi(v) \le \varphi(u)\} \notag \\
    \qquad \le 2 \sum_{v \in (T, S)_{u \downarrow} \colon s(v) \in (K', 3K']} \Pr\{\varphi(v) \le \varphi(u)\} . \eqlabel{union-bound}
  \end{align}
  Observe now that for $v = (u, j_1, \dots, j_\ell)$,
  \[
    \varphi(v) \le \varphi(u) \iff \prod_{i = 1}^{\ell} |(T, S)_{(u, j_1, \dots, j_i) \downarrow}| \ge \prod_{i = 1}^{\ell} (n - |(T, S)_{(u, j_1, \dots, j_i) \downarrow}|) .
  \]
  Observe that
  \begin{equation}
    \frac{|(T, S)_{(u, j_1, \dots, j_i) \downarrow}|}{n} \overset{d}{\to} B \prod_{m = 1}^i U_{j_m, m} , \eqlabel{conv-dist-sub}
  \end{equation}
  as $n \to \infty$, where each $U_{j_m, m}$ is an independent product
  of $j_m$ independent standard uniform random variables, and
  $B \sim \Beta(1, k - 1)$. So, after dividing through by $n$,
  \begin{align}
    &\Pr\{\varphi(v) \le \varphi(u)\} \eqlabel{failure-prob-v-r} \\
 &\qquad \le \Pr\left\{ \prod_{i = 1}^\ell B \prod_{m = 1}^i U_{j_m, m} \ge \prod_{i = 1}^\ell \left(1 - B \prod_{m = 1}^i U_{j_m, m} \right) \right\}  \notag \\
                                     &\qquad = \Pr\left\{ \prod_{i = 1}^\ell \prod_{m = 1}^i U_{j_m, m} \ge \prod_{i = 1}^\ell \left(\frac{1}{B} - \prod_{i = 1}^i U_{j_m, m} \right) \right\} \notag \\
                                     &\qquad \le \Pr\left\{ \prod_{i = 1}^\ell \prod_{m = 1}^i U_{j_m, m} \ge t \right\} +
 \Pr\left\{ \prod_{i = 1}^\ell \left(\frac{1}{B} - \prod_{m = 1}^i U_{j_m, m} \right) \le t \right\} , \eqlabel{lhs-rhs}
  \end{align}
  where $t > 0$ is to be specified later. The first inequality above
  follows by the portmanteau lemma and the convergence in distribution
  noted in \eqref{conv-dist-sub}.

  In \cite[Lemma 1]{finding-adam}, it is shown that
  \[
    \Pr\left\{ \prod_{i = 1}^\ell \prod_{m = 1}^i U_{j_m, m} \ge t \right\} \le \exp\left\{ - \sqrt{s/2} \log \left( \frac{s}{e \log (1/t) } \right) \right\} .
  \]
  On the other hand, observing that
  $1 - e^{-x} \ge (1/2) \min\{x, 1\}$ for all $x \ge 0$, we have
  \begin{align}
    &\Pr\left\{ \prod_{i = 1}^\ell \left(\frac{1}{B} - \prod_{m = 1}^i U_{j_m, m} \right) \le t \right\} \notag \\
    &\qquad \le \Pr\left\{ \prod_{i = 1}^\ell \left(\frac{1}{B} - 1 + \frac{1}{2} \min\left\{ \sum_{m = 1}^i \log(1/U_{j_m, m}), 1 \right\} \right) \le t \right\} \notag \\
        &\qquad \le \Pr\left\{ \prod_{i = 1}^\ell \left(\frac{1}{B} - 1 + \frac{1}{2} \min\left\{ \sum_{m = 1}^i E_m, 1 \right\}\right) \le t \right\} , \eqlabel{rhs-2}
  \end{align}
  where $E_1, E_2, \dots$ are independent standard exponential random
  variables. By the inequality of arithmetic and geometric means,
  \begin{align}
    \eqref{rhs-2} &\le \Pr\left\{ 2^{\ell/2} \left(\frac{1}{B} - 1\right)^{\ell/2} \left(\prod_{i = 1}^\ell \min\left\{ \sum_{m = 1}^i E_m, 1 \right\} \right)^{1/2} \le t \right\} \notag \\
        &\le \Pr\left\{ 2^\ell \left(\frac{1}{B} - 1\right)^\ell X \le t^2\right\} , \eqlabel{rhs-3}
  \end{align}
  where $X$ is defined by
  \[
    X = \prod_{i = 1}^\infty \min\left\{ \sum_{m = 1}^i E_m, 1 \right\} .
  \]
  Then, for $q > 1$ to be specified,
  \[
    \eqref{rhs-3} \le \Pr\left\{\frac{B}{1 - B} \ge \frac{2}{q^{1/\ell}} \right\} + \Pr\left\{X \le \frac{t^2}{q}\right\} .
  \]
  By \cite[Lemma 2]{finding-adam}, we know that
  \begin{equation}
    \Pr\left\{X \le \frac{t^2}{q}\right\} \le \frac{6 t^{1/2}}{q^{1/4}} , \eqlabel{x-tail}
  \end{equation}
  and it is also known that
  \begin{equation}
    \Pr\left\{ \frac{B}{1 - B} \ge \frac{2}{q^{1/\ell}} \right\} = \frac{1}{(1 + 2/q^{1/\ell})^{k - 1}} \le \exp\left\{- \frac{k - 1}{q^{1/\ell}}\right\} , \eqlabel{b-tail}
  \end{equation}
  where the inequality follows since $q > 1$, where here we have used
  that $\frac{1}{1 + x} \le e^{-x/2}$ for $0 \le x \le 2$.  Optimizing
  a choice of $q$ in \eqref{x-tail} against \eqref{b-tail}, one can
  see that when $t > (1/36) e^{-2 (k - 1)}$,
  \begin{align*}
    \eqref{rhs-3} &\le 12 t^{1/2} \left( \frac{\log(1/6t^{1/2}) + (\ell/4) \log \left( \frac{k - 1}{\log (1/6t^{1/2})} \right)}{k - 1} \right)^{\ell/4} \\
                &\le 12 t^{1/2} \left( \frac{\log(1/6t^{1/2}) + (\sqrt{2s}/4) \log \left( \frac{k - 1}{\log (1/6t^{1/2})} \right) }{k - 1} \right)^{\sqrt{2s}/4} ,
  \end{align*}
  where we used the fact that
  \[
    s = \sum_{i = 1}^\ell (\ell - i + 1) j_i \ge \ell^2/2 .
  \]
  
  It remains to make an optimal of choice of $t$. If we pick $t$ such
  that
  \[
    \log (1/6t^{1/2}) = s^{0.6} - (\sqrt{2s}/4) \log (k - 1) ,
  \]
  it can be shown that if $s > 10^{12}$, $t < 1/6^6$, and
  $\log(1/6t^{1/2}) > s^{0.6}/2$, then
  \[
    \eqref{failure-prob-v-r} \le 2 \exp\left\{ - \sqrt{s/2} \log \left( (1/25) s^{0.3} \log (k - 1) \right) \right\} . 
  \]

  Recall the union bound \eqref{union-bound},
  \[
    \Pr\{u_* \not\in H_{\varphi; K}(T^\circ)\} \le 2 \sum_{u \in V(S)} \left( \sum_{v \in (T, S)_{u \downarrow} \colon s(v) \in (K', 3K']} \Pr\{\varphi(v) \le \varphi(u)\} \right) .
  \]
  We know that for any $u \in V(S)$,
  \[
    |\{v \in (T, S)_{u \downarrow} \colon s(v) \in (K', 3K']\}| \le 3K'
    \exp\{ \pi \sqrt{2 K'} \} ,
  \]
  (see \cite[Page 8, Equation (6)]{finding-adam}) and by the
  conditions imposed on $s$ and $k$,
  \[
    6 k K' \exp \{\pi \sqrt{2 K'} \} \le 6 \exp\{ 11 \sqrt{K'} \} ,
  \]
  and therefore,
  \begin{align*}
    \Pr\{u_* \not\in H_{\varphi; K}(T^\circ)\} &\le 6 \exp\left\{ 11 \sqrt{K'} - \sqrt{K'/2} \log \left((1/25) K'^{0.3} \log (k - 1)\right) \right\} \\
                                               &\le 6 \exp\left\{ - (1/2) \sqrt{K'} \log \left((1/25) K'^{0.3} \log (k - 1) \right) \right\}
  \end{align*}
  for $K' > 10^{77}$, which holds for $k > 10^{10^8}$. In order to
  make this probability at most $\eps$, we can pick for some universal
  constant $c_1 > 0$,
  \[
    \sqrt{K'} = c_1  \frac{\log (1/\eps)}{\log \log (1/\eps) + \log \log k} .
  \]
  In order to satisfy the condition that
  $\log(1/6t^{1/2}) > s^{0.6}/2$, this requires that for some constant
  $c_2 > 0$,
  \[
    \eps < \exp \left\{ -c_2 (\log k)^{11} \right\} .
  \]
  In this case, there are universal constants $c_3, c_4 > 0$ such that
  \begin{align*}
    K &= |\{v \in (T, S)_{u_* \downarrow} \colon s(v) \le 3K'\} \\
      &\le c_3 \exp\left\{ c_4  \frac{\log (1/\eps)}{\log \log (1/\eps) + \log \log k} \right \} . \qedhere
  \end{align*}
\end{proof}

  \subsection{A lower bound}\seclabel{heart-lower}
  To obtain a lower bound on $K^*(k, \ell, \eps)$, one can use the
likelihood of an observation under $\UA(n, S)$, computed in
\secref{mle}, and as in \cite[Theorem 4]{finding-adam}, one can then
construct a family of probable trees whose maximum likelihood estimate
for $K$-sized sets to intersect the seed will avoid every node of the
seed---the right choice for $K$ so that such a tree appears with
probability at least $\eps$ would then give a lower bound on
$K^*(k, \ell, \eps)$. Instead of this lengthy retelling, we use
\lemref{conditional-indep} to show how any lower bound on $K(\eps)$
also offers a lower bound on $K^*(k, \ell, \eps)$.

Define $H^*_{K, k}(T^\circ)$ to be the maximum likelihood estimate for
the set of size $K$ most likely to contain at least one node of the
true seed of $T \sim \UA(n, S)$, given $|S| = k$ and $|L(S)| = \ell$:
\[
  H^*_{K, k, \ell}(T^\circ) = \argmax_{H^* \in V(T)^{(K)}} \sum_{S' \in
    \sS_{k, \ell}(T) \colon |V(S') \cap H^*| \ge 1} \sL_{T^\circ}(S') .
\]
In order to prove that $K^*(k, \ell, \eps) \ge K$ for some particular
$K$, it suffices to show that for some specific $n$, and for all $S$
with $|S| = k$ and $|L(S)| = \ell$,
\[
  \Pr_{T \sim \UA(n, S)}\{V(S) \cap H_{K, k, \ell}^*(T^\circ) = \emptyset\}
  \ge \eps .
\]

\begin{proof}[Proof of \thmref{heart-lower}]
  Let $m = k^2 K$. For brevity, let $M_u = |(T, S)_{u \downarrow}|$,
  and let $M_* = \min_{u \in S} M_u$. By \lemref{sukhatme},
  \[
    \frac{M_*}{n} \overset{d}{\to} \frac{\Beta(1, k - 1)}{k}
  \]
  as $n \to \infty$. Then,
  \begin{align*}
    \Pr_{T \sim \UA(m, S)} \{M_*/m > 1/k^2\} &\ge \liminf_{n \to \infty} \Pr_{T \sim \UA(n, S)} \{M_*/n > 1/k^2\} \\
                                             &\ge \Pr\{\Beta(1, k - 1) > 1/k\} \\
                                             &= (1 - 1/k)^{k - 1} \\
                                             &\ge e^{-1} .
  \end{align*}
  Let
  \[
    \sM = \left\{(m_u \colon u \in S) \colon m_u \in \N, \, \sum_{u \in S} m_u = m, \, \min_{u \in S} m_u > \frac{m}{k^2} \right\} .
  \]
  Upon conditioning,
  \begin{align*}
    &\Pr_{T \sim \UA(m, S)} \{V(S) \cap H_{K, k, \ell}^*(T^\circ) = \emptyset\} \\
    &\qquad \ge \Pr\{M_*/m > 1/k^2\} \\
    &\qquad \hphantom{\ge} \sum_{(m_u \colon u \in S) \in \sM} \Pr\left\{\bigcap_{u \in S} [u \not\in H_{K, k, \ell}^*(T^\circ) \cap (T, S)_{u \downarrow}] \; \middle| \; \bigcap_{u \in S} [M_u = m_u]\right\} \\
    &\qquad \ge e^{-1} \sum_{(m_u \colon u \in S) \in \sM} \left(\prod_{u \in S} \Pr\Bigl\{u \not\in H_{K, k, \ell}^*((T, S)_{u \downarrow}^\circ) \mathrel{\Big|} M_u = m_u \Bigr\} \right) \\
    &\qquad \ge e^{-1} \Bigl(\Pr_{T \sim \UA(K + 1)} \{u_1 \not\in H_{K, 1}^*(T^\circ)\}\Bigr)^k ,
  \end{align*}
  where this last line follows from the optimality of
  $H^*_{K, 1}((T, S)_{u \downarrow}^\circ)$ as a root estimator in
  $(T, S)_{u \downarrow}^\circ$, and since conditionally upon
  $M_u = m_u$, $(T, S)_{u \downarrow}$ is distributed as $\UA(m_u)$ by
  \lemref{conditional-indep}. By definition, if
  $K < K((e \eps)^{1/k})$, then the probability that
  $H^*_{K, k, \ell}(T^\circ)$ avoids $V(S)$ exceeds $\eps$.
\end{proof}

\corref{heart-lower} follows immediately.

\section{Finding all seed vertices}\seclabel{whole}
  \subsection{Upper bounds}\seclabel{whole-upper}
  \subsubsection{A familiar strategy}\seclabel{familiar}

We can also use the strategy $H^*_{\psi; K}$ of \cite{finding-adam}
and \secref{heart-upper} to get all the nodes of $S$; in this section,
when the procedure is used to find all nodes of the seed, we omit the
asterisk for notational consistency. More specifically, we study the
smallest choice of $K$ for which $H_{\psi; K}$ contains all nodes of
$S$ with probability at least $1 - \eps$, and such a choice will give
an upper bound on $K(k, \ell, \eps)$.

If $\psi(u) = |(T, u)_{v \downarrow}|$ for $v$ adjacent to $u$, we
will say that $\psi(u)$ is \emph{witnessed at $v$}.

\begin{prop}\proplabel{whole-upper}
  There are universal constants $c, \eps_0 > 0$ such that, if
  $\eps \le \eps_0$ and 
  \[
    K \ge (c k \ell/\eps) \log (k \ell/\eps) ,
  \]
  then
  \[
    \Pr\{V(S) \subseteq H_{\psi; K}(T^\circ)\} \ge 1 - \eps .
  \]
\end{prop}
\begin{proof}
  \begin{figure}
    \centering
    \includegraphics[width=0.6\textwidth]{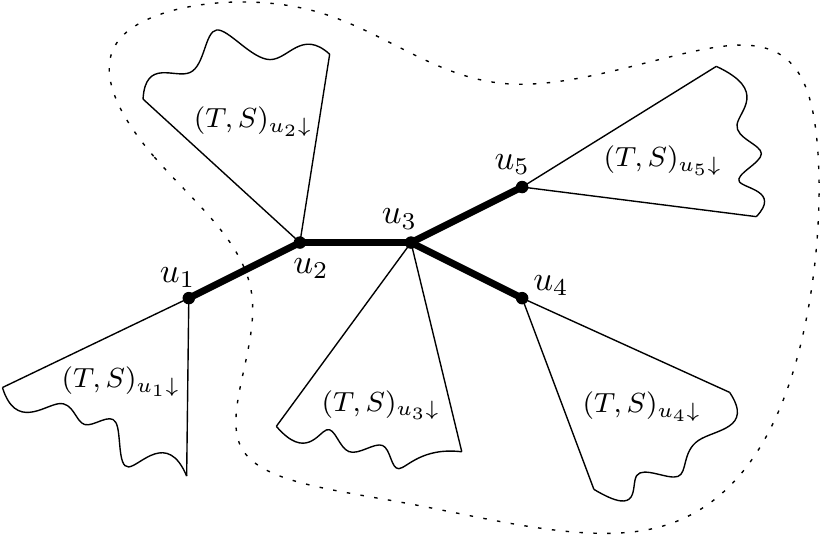}
    \caption{The seed $S$ has vertex set $\{u_1, \dots,
      u_5\}$. Suppose that $\psi^* = \psi(u_1)$. Then, $\psi^*$ is
      witnessed at $u_2$, and $\psi^*$ correponds to the size of the
      outlined subgraph.}
    \figlabel{whole-upper}
  \end{figure}
  The proof is similar to that of \propref{heart-upper}. Write
  \[
    \psi^* = \max\{\psi(u_1), \dots, \psi(u_k)\} .
  \]
  Observe that if for all $i > K$, $\psi(u_i) > \psi^*$, then
  $V(S) \subseteq H_{\psi; K}(T^\circ)$. So,
  \begin{equation}
    \Pr\{V(S) \not\subseteq H_{\psi; K}(T^\circ)\} \le \Pr\{\psi^* \ge n t \} + \Pr\{\exists i > K \colon \psi(u_i) \le n t\} \eqlabel{phi-tail-whole}
  \end{equation}
  for $t > 0$ to be specified. We handle the first term in
  \eqref{phi-tail-whole}: Suppose that $\psi^*$ is attained by
  $u \in V(S)$ and witnessed by its child $v \not\in V(S)$, \ie
  \[
    \psi^* = \psi(u) = |(T, u)_{v \downarrow}| ,
  \]
  Then, $u$ has a neighbour $w \in V(S)$, and
  \[
    \psi(w) \ge |(T, w)_{u \downarrow}| > |(T, u)_{v \downarrow}| = \psi(u) ,
  \]
  so $\psi(u)$ cannot be maximum. Thus, $\psi^*$ must be witnessed by
  a node of $V(S)$, and in particular, one can iterate the above
  motion to see that $\psi^*$ must be attained by a leaf of $S$ and
  witnessed by its unique neighbour in $S$, so
  \[
    \psi^* = \max_{u \in L(S)} \sum_{v \in V(S) - u} |(T, S)_{v \downarrow}| .
  \]
  See \figref{whole-upper} for an illustration. By
  \lemref{dir-convergence}, for any $u \in L(S)$,
  \[
    \frac{1}{n} \sum_{v \in V(S) - u} |(T, S)_{v \downarrow}| \overset{d}{\to} \Beta(k - 1, 1) ,
  \]
  as $n \to \infty$, so
  \begin{align*}
    \Pr\{\psi^* \ge nt\} &\le \Pr\left\{\exists u \in L(S) \colon \sum_{v \in V(S) - u} |(T, S)_{v \downarrow}| \ge n t \right\} \\
                         &\le \ell \Pr\{\Beta(k - 1, 1) \ge t\} \\
                         &= \ell (1 - t^{k - 1}) .
  \end{align*}
  We can make this at most $\eps/2$ by choosing
  $t = (1 - \eps/(2\ell))^{1/(k - 1)}$.

  For the second term in \eqref{phi-tail-whole}, the argument is again
  identical to that of \cite[Theorem 3]{finding-adam}, and we can say
  that
  \begin{align*}
    \Pr\{\exists i > K \colon \psi(u_i) \le nt\} \le K t^{K - 1} \le K e^{-\frac{\eps (K - 1)}{2 (k - 1)\ell}} .
  \end{align*}
  Picking $K \ge (c k \ell/\eps) \log (k\ell/\eps)$ for some constant
  $c > 0$ gives the desired result, as long as $\eps \le \eps_0$ for
  some constant $\eps_0 > 0$.
\end{proof}
Again, the set $H_{\psi; K}(T^\circ)$ can be computed in $O(n)$
time. We show in \lemref{leaf-existence} that the result of
\thmref{whole-upper} involves the right order for $K$ for the strategy
given by $H_{\psi; K}$, up to logarithmic factors: When
$K \le k \ell/(4 \eps)$, then with probability at least $\eps$, at
least one leaf of $S$ is also a leaf of $T \sim \UA(K, S)$, and any
leaf of $T$ maximizes the value of $\psi$.

\subsubsection{A reduction to intersection testing}

We now consider an alternative procedure for locating all nodes of the
seed, which sometimes requires fewer nodes than $H_{\psi; K}$ to
succeed with probability $1 - \eps$. We define the set
$H_{\phi; K, k, \ell, \eps}(T^\circ)$ as follows:
\begin{enumerate}[label=(\roman*)]
\item Let $H^*$ be a set of size $K^*(k, \ell, \eps/2)$ which
  intersects the seed with probability at least $1 - \eps/2$;
\item For each $u \in H^*$, traverse the tree $T$ in a depth-first
  manner around $u$;
\item When exploring $v \in T$, add $v$ to
  $H_{\phi; K, k, \ell, \eps}(T^\circ)$ if
  $|(T, u)_{v \downarrow}| \ge n\eps/(2 k\ell)$, and stop exploring
  this path otherwise. Stop at any point if the size of the set
  exceeds $K$.
\end{enumerate}
If we can prove that $H_{\phi; K, k, \ell, \eps}(T^\circ)$ includes
the whole seed with probability at least $1 - \eps$, and that $K$ is
sufficiently large, we obtain an upper bound on $K(k, \ell, \eps)$.
\begin{prop}\proplabel{whole-upper-ez}
  If $K \ge (2k \ell/\eps) K^*(k, \ell, \eps/2)$, then
  \[
    \Pr\{V(S) \subseteq H_{\phi; K, k, \ell, \eps}(T^\circ)\} \ge 1 -
    \eps .
  \]
\end{prop}
\begin{proof}
  We show the failure probability is small enough:
  \begin{align*}
    \Pr\{V(S) \not\subseteq H_{\phi; K, k, \ell, \eps}(T^\circ)\} \le \Pr\biggl\{V(S) \cap H^* = \emptyset\biggr\} + \Pr\biggl\{\min_{u \in L(S)} |(T, S)_{u \downarrow}| < \frac{n\eps}{2 k \ell}\biggr\} .
  \end{align*}
  The first term is at most $\eps/2$ by definition. For the second
  term, note that
  \[
    \frac{\min_{u \in L(S)} |(T, S)_{u \downarrow}|}{n} \overset{d}{\to} \frac{\Beta(1, k - 1)}{\ell} ,
  \]
  as $n \to \infty$ by \lemref{sukhatme}, so
  \begin{align*}
    \Pr\biggl\{\min_{u \in L(S)} |(T, S)_{u \downarrow}| < \frac{n \eps}{2k \ell}\biggr\} &\le \Pr\biggl\{\Beta(1, k - 1) \le \frac{\eps}{2k}\biggr\} \\
                                                                      &= 1 - \left(1 - \frac{\eps/2}{k}\right)^{k - 1} \\
                                                                                 &\le 1 - e^{-\eps/2} \\
                                                                                 &\le \eps/2 ,
  \end{align*}
  as desired. Finally, for $u$ fixed, there are at most
  $2 k \ell/\eps$ nodes $v$ such that
  $|(T, u)_{v \downarrow}| \ge n\eps/(2k\ell)$, so $K$ is indeed large
  enough to include all desired nodes.
\end{proof}
By a remark in \secref{mle}, the set $H^*$ in the above construction
can be computed in polynomial time, so
$H_{\phi; K, k, \ell, \eps}(T^\circ)$ can be computed in polynomial
time. \thmref{whole-upper} follows immediately from
\propref{whole-upper} and \propref{whole-upper-ez}.

  \subsection{Lower bounds}\seclabel{whole-lower}
  By \eqref{whole-harder-than-heart}, we have the same lower bound on
$K(k, \ell, \eps)$ as in \thmref{heart-lower}, \ie
\begin{cor}\corlabel{whole-lower-small-eps}
  There are universal constants $c_1, c_2, c_3 > 0$, and such that if
  $\eps \le e^{- c_1 k}$, then
  \[
    K(k, \ell, \eps) \ge c_2 \exp\left\{ c_3 \sqrt{\frac{\log
        (1/\eps)}{k}}\right\} .
  \]
\end{cor}

Unlike the case for $K^*(k, \ell, \eps)$, we do not know that
$K(k, \ell, \eps)$ is at most linear in $k$ when $\eps > e^{-c k}$;
our best upper bound from \thmref{whole-upper} says that
$K(k, \ell, \eps)$ is at most exponential in $k$, while the lower
bound from \corref{whole-lower-small-eps} does not apply. We thus
search for a lower bound on $K(k, \ell, \eps)$ which applies in the
regime when $\eps$ is large.

As in \secref{heart-lower}, define $H_{K, k, \ell}(T^\circ)$ to be the
set of size $K$ most likely to contain all the nodes of the true seed
of $T \sim \UA(n, S)$, given $|S| = k$ and $|L(S)| = \ell$:
\[
  H_{K, k, \ell}(T^\circ) = \argmax_{H \in V(T)^{(K)}} \sum_{S' \in \sS_{k, \ell}(T) \colon V(S') \subseteq H} \sL_{T^\circ}(S') .
\]

If $\eps$ is at least some positive constant, then
$K(k, \ell, \eps) \le c k \ell$ for some constant $c > 0$ by
\thmref{whole-upper}. We noted in \secref{familiar}, as a result of
\lemref{leaf-existence}, that if $H_{\psi; K}$ were the optimal
strategy for picking $K$ nodes to include the whole seed with
probability at least $1 - \eps$, then $k \ell/(4 \eps)$ is roughly a
lower bound on the size of such a vertex confidence set. We know that
$H_{\psi; K}$ is not in fact the optimal strategy, but one can
understand it to be a relaxation of the optimal strategy
$H_{K, k, \ell}$. We thus make the following conjecture, which
expresses our belief that $H_{\psi; K}$ is ``close enough'' to
$H_{K, k, \ell}$ when $\eps$ is large.
\begin{conj}\conjlabel{whole-lower-large-eps-conj}
  There are universal constants $c_1, c_2 , \eps_0 > 0$ such that if
  $k, \ell > c_1$, then,
  \[
    K(k, \ell, \eps_0) \ge c_2 k \ell .
  \]
\end{conj}

\section{Finding all leaves given the skeleton}\seclabel{leaves}

For a seed $S$, write $R(S) = S - L(S)$ for the \emph{skeleton} of
$S$, or simply $R$ when $S$ is understood. For an integer $i \ge 1$,
let $\sT_i$ denote the set of trees in which a set of $i$ labelled
vertices form a connected subgraph, and in which all other vertices
are unlabelled. For a given labelled tree $T$ and a set
$A \in V(T)^{(i)}$, let $T^{(A)} \in \sT_i$ be the tree in which all
labels are forgotten except for those of $A$. Consider now the optimal
size
\[
  K'(k, \ell, \eps) = \min\left\{ m \colon 
    \begin{array}{l} \displaystyle
      \exists H'_{m, k, \ell, \eps}
  \colon \sT_{k - \ell} \to V(\sT_{k - \ell})^{(m)} , \text{ such that} \\ \displaystyle
      \min\limits_{\substack{S \colon |S| = k \\ |L(S)| = \ell}} \Pr_{T
    \sim \UA(n, S)} \{L(S) \subseteq H'_{m, k, \ell, \eps}(T^{(R)})\}
  \ge 1 - \eps
      \end{array}
      \right\} .
\]
be the optimal size of a set which, given the position of the skeleton
of the seed, the size of the seed, and its number of leaves, will
locate all of its true leaves with probability at least $1 - \eps$.

As in \secref{heart} and \secref{whole}, we find an upper bound on
$K'(k, \ell, \eps)$ by exhibiting an algorithm which, given $k, \eps$
and $R$, returns a set of vertices which contains all of $L(S)$ with
probability at least $1 - \eps$.

Let $\psi(u) = |(T, R)_{u \downarrow}|$, and let
$H'_{\psi; K}(T^{(R)})$ be the set of $K$ vertices $u \in N(R)$
maximizing their value of $\psi$. This estimator for $L(S)$ is
slightly different than those for $V(S)$ seen in \secref{heart} and
\secref{whole}. Indeed, allowing ourselves to assume $R$ significantly
improves our chances at correctly guessing the rest of
$S$. Specifically, if $k$ is constant, the dependence of
$K'(k, \ell, \eps)$ upon $1/\eps$ is shown to be logarithmic, while
the result of \thmref{heart-lower} has that, for sufficiently small
$\eps$, $K^*(k, \ell, \eps)$ is superpolylogarithmic in $1/\eps$.

\begin{prop}\proplabel{given-skeleton-lem}
  If
  \[
    K \ge \ell + 2 (k - \ell) \log \left( (3 \ell/\eps) \log
      \left( \frac{3 (k - \ell)}{\eps}\right)\right) + (7/6)
    \log \left( \frac{3 (k - \ell)}{\eps} \right)
  \]
  then $\Pr\{L(S) \subseteq H'_{\psi; K}(T^{(R)})\} \ge 1 - \eps$.
\end{prop}
\begin{proof}
  Let $v_1, v_2, \dots$ be the chronological sequence of nodes
  attaching to $R$, where $\{v_1, v_2, \dots, v_\ell\} = L(S)$ ordered
  arbitrarily. Write $\psi_* = \min_{u \in L(S)} \psi(u)$. If for all
  $i > K$, $\psi(v_i) < \psi_*$, then
  $L(S) \subseteq H'_{\psi; K}(T^{(R)})$. So
  \begin{equation}
    \Pr\{L(S) \not\subseteq H'_{\psi; K}(T^{(R)})\} \le \Pr\{\psi_* \le t n\} + \Pr\{ \exists i > K \colon \psi(v_i) \ge t n\} \eqlabel{phi-tail-skeleton}
  \end{equation}
  for $t > 0$ to be specified. Observe that $\psi_*/n$ converges in
  distribution to $\Beta(1, k - 1)/\ell$ as $n \to \infty$, so
  \begin{align*}
    \Pr\{\psi_* \le tn\} &\le \Pr\{\Beta(1, k - 1) \le t \ell \} \\
                         &\le t \ell (k - 1) \\
                         &\le \eps/3
  \end{align*}
  if we choose $t = \eps/(3 \ell (k - 1))$ .

  It remains to handle the second term in
  \eqref{phi-tail-skeleton}. Let $N$ be the (random) time at which
  $v_{K}$ is inserted, \ie $v_{K} = u_N$. Let $T_u$ be the component
  of $T$ containing $u$ after the removal of all edges between
  $\{u_1, \dots, u_N\}$. Any node $v_i$ with $i > K$ is part of $T_u$
  for some $u \in R$. Since for any $u \in R$, $|T_u|/n$ converges in
  distribution to $\Beta(1, N - 1)$ as $n \to \infty$,
  \begin{align}
    &\Pr\{\exists i > K \colon \psi(v_i) \ge tn\} \notag \\
    &\qquad \le \Pr\{\exists u \in R \colon |T_u| \ge n t\} \notag \\
                                                  &\qquad \le (k - \ell) \Pr\{ \Beta(1, N - 1) \ge t \} \notag \\
                                                  &\qquad \le (k - \ell) \Pr\{\Beta(1, N - 1) \ge t \mid N \ge s\} + (k - \ell)\Pr\{N \le s\} , \eqlabel{skeleton-2}
  \end{align}
  where $s > 0$ is to be specified. Conditionally upon $N \ge s$,
  $\Beta(1, s - 1)$ stochastically dominates $\Beta(1, N - 1)$ by
  \lemref{beta-dom-2}, so
  \begin{align*}
    (k - \ell) \Pr\{\Beta(1, N - 1) \ge t \mid N \ge s\} &\le (k - \ell) \Pr\{\Beta(1, s - 1) \ge t\} \\
    &= (k - \ell) (1 - t)^{s - 1} \\
    &\le (k - \ell) e^{- t (s - 1)} , 
  \end{align*}
  which is less than $\eps/3$ if we choose
  \[
    s = 1 + (1/t) \log \left(\frac{3 (k - \ell)}{\eps}\right) = 1 + \frac{3 \ell (k -
      1)}{\eps} \log \left(\frac{3 (k - \ell)}{\eps}\right) .
  \]
  
  For the second term in \eqref{skeleton-2}, observe that
  \[
    \Pr_{T \sim \UA(n, S)} \{N \le s\} \le \Pr_{T \sim \UA(s, S)}\{\deg(R) \ge K \} ,
  \]
  where
  \[
    \deg(R) = \ell + \sum_{i = k + 1}^s \mathbf{1}\{u_i \text{ connects to $R$} \} 
  \]
  and each such indicator is independent. Clearly, for any
  $i \ge k + 1$,
  \[
    \mathbf{1}\{u_i \text{ connects to $R$}\} \sim \Ber\left( \frac{k -
        \ell}{i - 1} \right) ,
  \]
  so writing $H_m = \sum_{i = 1}^m 1/i$ for the $m$-th Harmonic
  number,
  \[
    \E\{\deg(R)\} = \ell + (k - \ell)(H_{s - 1} - H_{k - 1}) .
  \]
  Picking
  \[
    K \ge \ell + (k - \ell) (H_{s - 1} - H_{k - 1}) + \delta
  \]
  for $\delta > 0$ to be specified, we have by Bernstein's
  inequality~\cite{bernstein, gabor-concentration},
  \begin{align*}
    &\Pr\{\deg(R) \ge K\} \\
    &\qquad \le \Pr \left\{ \left[ \sum_{i = k + 1}^s \Ber\left(\frac{k - \ell}{i - 1}\right) \right] - (k - \ell)(H_{s - 1} - H_{k - 1}) \ge \delta \right\} \\
    &\qquad \le \exp \left\{ - \frac{\delta^2}{2 (k - \ell) (H_{s - 1} - H_{k - 1}) + 2 \delta/3} \right\} \\
    &\qquad \le \exp \left\{ - \frac{\delta^2}{2 (k - \ell) \log \left(\frac{s - 1}{k}\right) + 2 \delta/3} \right\} .
  \end{align*}
  Some arithmetic shows that picking
  \[
    \delta = (2/3) \log \left( \frac{3 (k - \ell)}{\eps} \right) + \sqrt{2 (k - \ell) \log \left(\frac{s - 1}{k}\right) \log \left(\frac{3(k - \ell)}{\eps}\right)}
  \]
  suffices to have
  \[
    \Pr\{\deg(R) \ge K\} \le \frac{\eps}{3(k - \ell)} ,
  \]
  and therefore, in \eqref{skeleton-2},
  \[
    (k - \ell) \Pr\{N \le s\} \le \eps/3 .
  \]
  With our particular choice of $s$, this proves that it suffices to pick
  \[
    K \ge \ell + 2 (k - \ell) \log \left( (3 \ell/\eps) \log \left( \frac{3 (k - \ell)}{\eps}\right)\right) + (7/6) \log \left( \frac{3 (k - \ell)}{\eps} \right) . \qedhere
  \]
\end{proof}
\thmref{skeleton-tight} follows immediately. In fact,
\propref{given-skeleton-lem} is tight for a large class of seeds. In
order to prove this, we use the following basic result.
\begin{lem}\lemlabel{leaf-existence}
  When $K \le k \ell/(4 \eps)$,
  \[
    \Pr_{T \sim \UA(K, S)}\{\exists u \in L(S) \colon |(T, S)_{u \downarrow}| = 1 \} \ge \eps .
  \]
\end{lem}
\begin{proof}
  Since $S$ is a tree, it has at least two leaves. Write $\sE_u$ for
  the event that $|(T, S)_{u \downarrow}| = 1$. By
  inclusion-exclusion, for some arbitrary distinct leaves
  $u, v \in L(S)$,
  \begin{align*}
    \Pr\{\exists w \in L(S) \colon \sE_w\} \ge \ell \Pr\{\sE_u\} - \binom{\ell}{2} \Pr\{\sE_u \cap \sE_v\} .
  \end{align*}
  It is easy to see that
  \[
    \Pr\{\sE_u\} = \frac{k - 1}{k} \cdot \frac{k}{k + 1} \cdots \frac{K - 2}{K - 1} = \frac{k - 1}{K - 1} ,
  \]
  and
  \[
    \Pr\{\sE_u \cap \sE_v\} = \frac{k - 2}{k} \cdot \frac{k - 1}{k + 1} \cdots \frac{K - 3}{K - 1} = \frac{(k - 2)(k - 1)}{(K - 2)(K - 1)} \le \left( \frac{k - 1}{K - 1} \right)^2 ,
  \]
  so
  \[
    \Pr\{\exists w \in L(S) \colon \sE_w\} \ge \frac{k \ell}{4 K} . \qedhere
  \]
\end{proof}

Once again, we rely on the maximum likelihood estimate to prove a
lower bound. Let $H'_{K, k, \ell}(T^{(R)})$ be the maximum likelihood
estimate for $K$-sized sets to include all leaves of a seed $S$, given
the skeleton $R$, and given $|S| = k$ and $|L(S)| = \ell$, \ie
\[
  H'_{K, k, \ell}(T^{(R)}) = \argmax_{H \in N(R)^{(K)}} \sum_{L' \subseteq H \colon |L'| = \ell}
  \sL_{T^{(R)}}(L') ,
\]
where $\sL_{T^{(R)}}(L')$ represents the likelihood of observing the
tree $T^{(R)}$ if it were drawn from $\UA(n, R \cup L')^{(R)}$.
\begin{prop}\proplabel{skeleton-tight}
  Let $\ell_2 = |L(R)|$. Suppose that
  $k - \ell - \ell_2 \ge 2 \sqrt{k}$, $\ell_2 \ge 2 \sqrt{k}$, and
  \[
    \eps \le \frac{\ell}{128 e^5 \ell_2^4} .
  \]
  Then, there is a universal constant $c > 0$ such that if
  \[
    K \le  c (k - \ell - \ell_2) \log (\ell_2 \ell/\eps) ,
  \]
  then
  \[
    \Pr_{T \sim \UA(n, S)}\{L(S) \subseteq H'_{K, k, \ell}(T^{(R)})\}
    < 1 - \eps .
  \]
\end{prop}
\begin{proof}
  Let $\sE_u$ be the event that $|(T, S)_{u \downarrow}| = 1$. By
  \lemref{leaf-existence}, when $n = k \ell/(64\eps)$,
  \[
    \Pr_{T \sim \UA(n, S)}\{ \exists u \in L(S) \colon \sE_u\} \ge 16 \eps .
  \]
  Let
  \[
    X = |\{v \in N(R - L(R)) - S \colon |(T, S)_{u \downarrow}| \ge 2\}| .
  \]
  For $u \in L(R)$, let $\sF_u$ be the event that $|N(u) - S| \ge 1$,
  and let $\sF = \cap_{u \in L(R)} \sF_u$. Then,
  \begin{align*}
    &\Pr\{L(S) \not\subseteq H'_{K, k, \ell}(T^{(R)})\} \\
    &\qquad \ge 16 \eps \Pr\bigl\{L(S) \not\subseteq H_{K, k, \ell}(T^{(R)}) \mathrel{\bigm|} \exists u \in L(S) \colon \sE_u\bigr\} \\
    &\qquad \ge 16 \eps \Pr\bigl\{[X \ge K] \cap \sF \mathrel{\bigm|} \exists u \in L(S) \colon \sE_u\bigr\} \\
    &\qquad \ge 16 \eps \Pr\bigl\{[X \ge K] \cap \sF \bigr\} ,
  \end{align*}
  where the second inequality follows since the $X$ vertices $u$ of
  $N(R - L(R)) - S$ with $|(T, S)_{u \downarrow}| \ge 2$ are more
  likely than at least one vertex of $L(S)$ with
  $|(T, S)_{u \downarrow}| = 1$ of being chosen in
  $H'_{K, k, \ell}(T^{(R)})$ when $\sF$ holds, and the third
  inequality follows since conditioning on the seed's leaves to be
  naked can only increase the likelihood of connections to skeleton
  nodes. It suffices to prove that $\Pr\{[X \ge K] \cap \sF\} \ge 1/16$.

  Let
  \[
    T_R = \bigcup_{u \in R - L(R)} (T, S)_{u \downarrow}, \quad T_L = \bigcup_{u \in L(R)} (T, S)_{u \downarrow} ,
  \]
  with sizes $|T_R| = M_R$ and $|T_L| = M_L$.  Observe that,
  conditionally upon the sizes $M_R$ and $M_L$, the events $X \ge K$
  and $\sF$ are independent. Furthermore,
  \[
    \frac{M_R}{n} \overset{d}{\to} B_R \sim \Beta(k - \ell - \ell_2, \ell + \ell_2), \quad \frac{M_L}{n} \overset{d}{\to} B_L \sim \Beta(\ell_2, k - \ell_2) ,
  \]
  as $n \to \infty$, where we note that $B_R$ and $B_L$ are not
  necessarily independent. Let $\sM \subseteq \N^2$ be such that for
  all $(m_R, m_L) \in \sM$,
  \[    
    \left|\frac{m_R}{n} - \frac{k - \ell - \ell_2}{k}\right| <
    \frac{1}{\sqrt{k}} \quad \text{ and } \quad \left|\frac{m_L}{n} - \frac{\ell_2}{k}\right| <
    \frac{1}{\sqrt{k}} .
  \]
  Write $M = (M_R, M_L)$ for brevity. By \lemref{beta-conc},
  \begin{align*}
    \Pr\{M \in \sM\} &= \Pr\left\{\left|\frac{M_R}{n} - \frac{k - \ell - \ell_2}{k}\right| < \frac{1}{\sqrt{k}} , \,  \left|\frac{M_L}{n} - \frac{\ell_2}{k}\right| < \frac{1}{\sqrt{k}}\right\} \\
    &\ge \Pr\left\{\left| B_R - \frac{k - \ell - \ell_2}{k}\right| < \frac{1}{\sqrt{k}} , \, \left|B_L - \frac{\ell_2}{k}\right| < \frac{1}{\sqrt{k}}\right\} \\
    &\ge 1 - \Pr\left\{\left|B_R - \frac{k - \ell - \ell_2}{k}\right| \ge \frac{1}{\sqrt{k}}\right\} - \Pr\left\{\left|B_L - \frac{\ell_2}{k}\right| \ge \frac{1}{\sqrt{k}}\right\} \\
    &\ge 1/2 .
  \end{align*}
  Therefore,
  \begin{align*}
    &\Pr\{[X \ge K] \cap \sF\} \\
    &\qquad \ge \sum_{m \in \sM} \Pr\bigl\{[X \ge K] \cap \sF \mathrel{\bigm|} M = m\bigr\} \Pr\{M = m\} \\
    &\qquad = \sum_{m \in \sM} \Pr\bigl\{X \ge K \mathrel{\bigm|} M = m\bigr\} \Pr\bigl\{\sF \mathrel{\bigm|} M = m\bigr\} \Pr\{M = m\} \\
    &\qquad = \sum_{m \in \sM} \Pr_{T \sim \UA(n, S)}\{M = m\} \\
    &\qquad \qquad \Pr_{T \sim \UA(m_R, R - L(R))}\{\deg(R - L(R)) \ge K\} \\
    &\qquad \qquad \Pr_{T \sim \UA(m_L, L(R))}\{\forall u \in L(R), \, \deg(u) \ge 1\} \\
    &\qquad \ge (1/2) \Pr_{T \sim \UA(n(k - \ell - \ell_2 - \sqrt{k})/k, R - L(R))}\{\deg(R - L(R)) \ge K\} \\
    &\qquad \qquad \Pr_{T \sim \UA(n(\ell_2 - \sqrt{k})/k, L(R))}\{\forall u \in L(R), \, \deg(u) \ge 1\} .
  \end{align*}
  Let $v_1, \dots, v_{n(k - \ell - \ell_2 - \sqrt{k})/k}$ be the sequence of
  nodes connecting to $R - L(R)$ in the uniform attachment process
  implied by the first probability above. Let
  \[
    X_i = \mathbf{1}\{v_i \text{ connects to } R \text{ and } v_j \text{ connects to } u_i \text{ for some } j > i\} .
  \]
  Then,
  \[
    \deg(R - L(R)) = \sum_{i = k - \ell - \ell_2 + 1}^{n(k - \ell - \ell_2 - \sqrt{k})/k} X_i ,
  \]
  where
  $\{X_i \colon k - \ell - \ell_2 + 1 \le i \le n (k - \ell - \ell_2 -
  \sqrt{k})/k\}$ is a collection of independent Bernoulli random variables
  with
  \[
    \E\{X_i\} = (k - \ell - \ell_2)\left(\frac{1}{i - 1} - \frac{1}{\frac{n(k - \ell - \ell_2 - \sqrt{k})}{k} - 1} \right) .
  \]
  Since $k - \ell - \ell_2 \ge 2\sqrt{k}$ by assumption, we can see
  that
  \[
    \E\{\deg(R - L(R))\} \ge (k - \ell - \ell_2) \log \left(\frac{n}{2 e^2 k}\right) .
  \]
  Since each $X_i$ is independent, then
  \begin{align*}
    \V\{\deg(R - L(R))\} &= \sum_{i = k - \ell - \ell_2 + 1}^{n(k - \ell - \ell_2 - \sqrt{k})/k} \V\{X_i\} \\
                         &\le \sum_{i = k - \ell - \ell_2 + 1}^{n(k - \ell - \ell_2 - \sqrt{k})/k} \E\{X_i\} \\
                         &= \E\{\deg(R - L(R))\} .
  \end{align*}
  Since the median of $\deg(R - L(R))$ is within a standard deviation
  of its mean, we see that picking
  \[
    K \le \left(\frac{k - \ell - \ell_2}{2}\right) \log
    \left(\frac{n}{2 e^2 k} \right)
  \]
  is sufficient to make
  \[
    \Pr_{T \sim \UA(n(k - \ell - \ell_2 - \sqrt{k})/k, R - L(R))}\{\deg(R - L(R)) \ge K\} \ge 1/2 ,
  \]
  as long as
  $(k - \ell - \ell_2) \log \left(\frac{n}{2 e^2 k}\right) \ge
  4$. Since $k - \ell - \ell_2 \ge 2 \sqrt{k} \ge 2$, this condition
  is satisfied as long as $n \ge e^4 k$. This condition will be
  absorbed by a further condition on $n$, and can be safely ignored.

  Let $w_1, \dots, w_{n(\ell_2 - \sqrt{k})/k}$ be the chronological sequence
  of nodes appearing in the attachment process
  $T \sim \UA(n(\ell_2 - \sqrt{k})/k, L(R))$, and let $z_1, \dots, z_Y$ be
  the subsequence of these nodes which connect directly to the nodes
  of $L(R)$. Finally, let $Z$ be minimum such that the nodes
  $z_1, \dots, z_Z$ connect to all nodes of $L(R)$. Then,
  \[
    \Pr_{T \sim \UA(n(\ell_2 - \sqrt{k})/k, L(R))}\{\forall u \in L(R), \, \deg(u) \ge 1\} \ge \Pr\{Z \le Y\} .
  \]
  Then, writing
  \[
    Y_i = \mathbf{1}\{w_i \text{ connects to a node of } L(R)\} ,
  \]
  then we see that
  $\{Y_i \colon \ell_2 + 1 \le i \le n (\ell_2 - \sqrt{k})/k\}$ is a
  collection of independent Bernoulli random variables such that
  $Y_i \sim \Ber\left(\frac{\ell_2}{i - 1}\right)$ and
  \[
    Y = \sum_{i = \ell_2 + 1}^{n (\ell_2 - \sqrt{k})/k} Y_i .
  \]
  Just as before, we see that if $\ell_2 \ge 2 \sqrt{k}$, then
  $\E\{Y\} \ge \ell_2 \log \left(\frac{n}{2 e k}\right)$, and
  \[
    \Pr\left\{Y \ge (\ell_2/2) \log\left(\frac{n}{2ek}\right)\right\} \ge 1/2 ,
  \]
  whence
  \begin{align*}
    \Pr\{\forall u \in L(R), \, \deg(u) \ge 1\} &\ge (1/2) \Pr\left\{Z \le Y \mathrel{\Big|} Y \ge (\ell_2/2) \log\left( \frac{n}{2ek} \right) \right\} \\
                                                &\ge (1/2) \Pr\left\{ Z \le (\ell_2/2) \log \left( \frac{n}{2 e k} \right) \right\} .
  \end{align*}
  The random variable $Z$ is distributed as the time to collect all
  coupons in the well-known \emph{coupon collector}
  problem~\cite[Section~2.4.1]{upfal}. Specifically,
  \[
    Z \sim \sum_{i = 1}^{\ell_2} \Geo\left(\frac{\ell_2 - (i - 1)}{\ell_2} \right) ,
  \]
  where each term above is independent, and where $\Geo(p)$ denotes a
  geometric random variable with parameter $p$, \ie the random
  variable with probability mass function $f_p \colon \N \to \R$,
  where
  \[
    f_p = p (1 - p)^{k - 1} .
  \]
  Then, by Markov's inequality,
  \begin{align*}
    \Pr\left\{Z \le (\ell_2/2) \log\left( \frac{n}{2ek} \right) \right\} &= 1 - \Pr\left\{Z > (\ell_2/2) \log\left( \frac{n}{2 e k} \right) \right\} \\
                                                                         &\ge 1 - \frac{\ell_2 \log (e \ell_2)}{(\ell_2/2) \log \left( \frac{n}{2 e k} \right)} \\
                                                                         &\ge 1/2
  \end{align*}
  as long as $n \ge 2 e^5 \ell_2^4 k$. Finally, this proves that
  \[
    \Pr\{[X > K] \cap \sF\} \ge 1/16 . \qedhere
  \]
\end{proof}
Note that the family of complete binary trees satisfies the structural
seed conditions of \propref{skeleton-tight}. Indeed, if $S$ is a
complete binary tree, we have that $k = 2 \ell - 1$ and
$\ell = 2 \ell_2 - 1$, so that
\[
  k - \ell - \ell_2 = k - \frac{k + 1}{2} - \frac{k + 3}{4} = \frac{k - 5}{4} \ge 2 \sqrt{k} ,
\]
and
\[
  \ell_2 = \frac{k + 3}{4} \ge 2 \sqrt{k} ,
\]
for all $k \ge 74$.

\section{Finding a whole star}\seclabel{whole-star}

In this section, we study the number of nodes $K(S_k, \eps)$ required
to find the seed in a seeded uniform attachment tree with probability
at least $1 - \eps$, given that the seed is a star $S_k$ on $k$ nodes,
and given that we know that the seed is isomorphic to $S_k$. By
\eqref{whole-harder-than-given} and \thmref{whole-upper},
\[
  K(S_k, \eps) \le c k^2 (1/\eps)^{1 + 2/k} \log(1/\eps) .
\]
The extra knowledge of the full structure of $S_k$ allows us to shave
off an extra factor of $k/\eps$.

Let $u_1$ denote the center of the star. To identify the whole star,
we first locate $u_1$ and then use \propref{given-skeleton-lem} to
locate all remaining vertices. Recall $H^*_{\psi; m}$ from
\secref{heart-upper}. We use the following intermediate result, whose
proof is adapted from that of \thmref{heart-upper}.
\begin{lem}\lemlabel{include-center-star}
  There are universal constants $c, \eps_0 > 0$ such that if
  $\eps \le \eps_0$ and
  \[
    K \ge c (1/\eps)^{1/k} \log (1/\eps) ,
  \]
  then
  \[
    \Pr\{u_1 \in H^*_{\psi; K}(T^\circ)\} \ge 1 - \eps .
  \]
\end{lem}
\begin{proof}
  If for all $i > K$, $\psi(u_i) > \psi(u_1)$, then
  $H^*_{\psi; K}(T^\circ)$ contains $u_1$. Moreover, if $T_i$ denotes
  the component of $T$ containing $u_i$ after the removal of all edges
  between vertices of $S_k$,
  \begin{align*}
    \Pr\{\psi(u_1) \ge nt\} &\le \Pr\{\exists 1 \le i \le k \colon |T_i| \ge nt\} \\
                            &\le k \Pr\{ \Beta(1, k - 1) \ge t\} \\
                            &= k (1 - t)^{k - 1}
  \end{align*}
  and this probability is at most $\eps/2$ for
  $t = 1 - (\eps/(2k))^{1/(k - 1)}$. As before,
  \[
    \Pr\{\exists i > K \colon \psi(u_i) \le nt\} \le K t^{K - 1} \le K e^{-(K - 1) \left(\frac{\eps}{2k}\right)^{1/(k - 1)}},
  \]
  and it is not hard to see that this probability can be made at most
  $\eps/2$ by choosing, for some constant $c > 0$,
  \[
    K \ge c (1/\eps)^{1/k} \log (1/\eps) . \qedhere
  \]
\end{proof}
We note here a related result by Jog and
Loh~\cite[Theorem~4]{jog-persistence}, which says that for a
universal constant $c > 0$, if $k \ge c \log (1/\eps)$, then with
probability at least $1 - \eps$, the node $u_1$ will be the unique
\emph{persistent centroid} of $\UA(n, S_k)$, \ie for sufficiently
large $n$, $u_1$ will minimize the value of $\psi$ in
$T \sim \UA(n, S_k)$, and remain as such throughout the rest of the
attachment process. As a consequence, we can find $u_1$ by selecting
only one node in the unlabelled tree. To summarize, when
$k \ge c \log (1/\eps)$,
\[
  \Pr\{u_1 \in H^*_{\psi; 1}(T^\circ)\} \ge 1 - \eps .
\]

Recall also $H'_{\psi; m}$ from \secref{leaves}.
\begin{prop}\proplabel{find-all-star}
  Let $m$ be such that
  \[
    \Pr_{T \sim \UA(n, S_k)}\{u_1 \in H^*_{\psi; m}(T^\circ)\} \ge 1 - \eps/2 
  \]
  and $m'$ be such that
  \[
    \Pr_{T \sim \UA(n, S_k)} \{ L(S_k) \subseteq H'_{\psi;
      m'}(T^{(u_1)})\} \ge 1 - \eps/2 .
  \]
  Then $K(S_k, \eps) \le m m'$.
\end{prop}
\begin{proof}
  Write $m' = K'(k, k - 1, \eps)$. Define
  \[
    H = \{v \colon v \in H'_{\psi; m'}(T^{(u)}) \text{ for all } u \in H^*_{\psi; m}(T^\circ)\} .
  \]
  Then,
  \[
    \Pr\{ V(S_k) \not\subseteq H \} \le \Pr\{u_1 \not\in H^*_{\psi; m}(T^\circ)\} + \Pr\{L(S_k) \not\subseteq H'_{\psi; m'}(T^{(u_1)})\} \le \eps 
  \]
  and clearly $|H| \le m m'$.
\end{proof}
\thmref{different-setting} follows from \lemref{include-center-star},
\propref{given-skeleton-lem}, and \propref{find-all-star}.

\section{Open problems}
Our work raises several open problems.
\begin{enumerate}[label=\textit{\arabic*.}]
\item \textit{Joint dependence on $k$ and $\eps$.} From \thmref{heart-upper}, we learn that $K^*(k, \ell, \eps)$
  grows roughly like $e^{1/k}$ for fixed $\eps$, and like
  $\mathrm{poly}(1/\eps)$ for fixed $k$;
  \thmref{heart-upper-subpolynomial} tells us that
  $K^*(k, \ell, \eps)$ grows like $e^{1/\log \log k}$ for fixed
  $\eps$, and
  $\exp\left\{\frac{\log (1/\eps)}{\log \log (1/\eps)}\right\}$ for
  fixed $k$. Can we find an upper bound on $K^*(k, \ell, \eps)$ which
  jointly behaves well as a function of $\eps$ and $k$, like
  \[
    K^*(k, \ell, \eps) \overset{?}{\le} c_1 \exp\left\{c_2 \frac{\log (1/\eps)}{\log \log (1/\eps) + k} \right\} .
  \]
  Should there be some dependence on $\ell$?

\item \textit{Tight bounds for $K$ and $K^*$.} What is the true dependence of $K(k, \ell, \eps)$ and
  $K^*(k, \ell, \eps)$ on $k$, $\ell$, and $\eps$? In particular, we
  ask
  \[
    K^*(k, \ell, \eps) \overset{?}{\le} c_1 \exp\left\{c_2 \sqrt{\frac{\log
          (1/\eps)}{k}} \right\} .
  \]
  This question remains open even for $k = 1$, where the best and only
  known result is from \cite{finding-adam}:
  \[
    c_1 \exp\left\{ c_2 \sqrt{\log (1/\eps)} \right\}
    \le K(\eps) \le c_3 \exp\left\{c_4 \frac{\log (1/\eps)}{\log
          \log(1/\eps)} \right\} .
  \]

\item \textit{Lower bounds for constant $\eps$.} Restating
  \conjref{whole-lower-large-eps-conj}, we ask: Can it be shown that
  for constants $c, \eps_0$ and sufficiently large $k$ and $\ell$,
  \[
    K(k, \ell, \eps_0) \overset{?}{\ge} c k \ell .
  \]
  
\item \textit{Partial vertex-confidence sets.}
  What about the optimal quantities $K^i(k, \ell, \eps)$ for the
  smallest sets which intersect at least $i$ nodes of seed with
  probability at least $1 - \eps$, where $1 \le i \le k$? It is clear
  that
  \[
    K^*(k, \ell, \eps) \le K^i(k, \ell, \eps) \le K(k, \ell, \eps) .
  \]
  Can this obvious result be refined?

\item \textit{The preferential attachment model.} Can one prove
  analogous upper and lower bounds on $K^i(k, \ell, \eps)$ in the
  seeded (superlinear/sublinear) preferential attachment tree
  $\UA_\alpha(n, S)$ for $\alpha > 0$? Jog and
  Loh~\cite{ling-centrality} showed that, for a given $\eps$, there
  are $c, N$ depending on $\eps$ such that for $K \ge N$ satisfying
  \[
    \frac{c K (\log K)^{\frac{2}{1 - \alpha}}}{(K - 1)2} \le \frac{\eps}{4} ,
  \]
  there exists a vertex-confidence set of size $K$ which includes the
  root in $\UA_\alpha(n)$ with probability at least $1 - \eps$.

\item \textit{Worst-case gnostic seed recovery.} We showed how when
  the seed is known to be a star $S_k$, only a constant factor of $k$
  nodes were required to recover the seed with probability at least
  $1/2$. Is there any seed $S$ for which $K(S, 1/2) = \omega_k(k)$,
  and in general what is the dependence on $k$ of
  \[
    \max_{S \colon |S| = k} K(S, 1/2) \, ?
  \]
  Any seed with $K(S, 1/2) = \omega_k(k)$ must have
  $\ell = \omega_k(1)$, since by \eqref{whole-harder-than-given} and
  \thmref{whole-upper}, 
  \[
    K(S, 1/2) \le K(k, \ell, 1/2) \le c k \ell .
  \]
  As natural candidates, we suggest that $S$ is a complete binary
  tree, or a \emph{comb graph}, namely that
  $V(S) = \{u_1, v_1, u_2, v_2, \dots, u_{k/2}, v_{k/2}\}$, and
  \[
    E(S) = \Big\{\{u_i, u_{i + 1}\} \colon 1 \le i \le k/2 - 1\}\Big\} \cup \Big\{\{u_i, v_i\} \colon 1 \le i \le k \Big\} .
  \]
\end{enumerate}

\bibliography{main}
\bibliographystyle{abbrv}

\appendix

\begin{appendices}

\section{Supporting lemmas}\applabel{app}
\begin{lem}[Sukhatme~\cite{luc-variate-generation, sukhatme}]\lemlabel{sukhatme}
  Let $U_1, \dots, U_{k - 1}$ be independent identically distributed
  $\text{Uniform}[0, 1]$ random variables, where $U_{(i)}$ denotes the
  $i$-th smallest among $U_1, \dots, U_{k - 1}$. Define the spacings
  $S_i = U_{(i)} - U_{(i - 1)}$, where $U_{(0)} = 0$ and
  $U_{(k)} = 1$. Then,
  $(S_1, \dots, S_k) \sim \Dirichlet(1, \dots, 1)$. Moreover, for
  independent identically distributed standard exponential random
  variables $E_1, \dots, E_k$,
  \[
    S_i \sim \frac{E_i}{\sum_{i = 1}^k E_i} \sim\Beta(1, k - 1)
  \]
  for each $1 \le i \le k$, and in particular, if $I$ is some index
  set of size $j$ for $1 \le j \le k$, then
  \[
    \min_{i \in I} S_i \sim \frac{\Beta(1, k - 1)}{j} .
  \]
\end{lem}

\begin{lem}\lemlabel{dir-convergence}
  Let $j \ge k$, and let $T_u$ be the subtree of $T \sim \UA(n, S)$
  containing the vertex labelled $u$ after removing all edges between
  vertices $\{u_1, \dots, u_j\}$. Then, as $n \to \infty$,
  \[
    \frac{1}{n} (|T_{u_i}| \colon 1 \le i \le j) \overset{d}{\to} \Dirichlet(\underbrace{1, \dots, 1}_{\text{$j$ times}}) .
  \]
\end{lem}
\begin{proof}
  It suffices to show that the vector
  $(|T_{u_i}| \colon 1 \le i \le j)$ evolves as a P\'{o}lya urn with
  $j$ colours, starting with one ball of each colour, and with
  replacement matrix $I_j$, the $j \times j$ identity
  matrix~\cite{blackwell}. Indeed, at each step in the attachment
  process wherein the node $u_n$ is attached, it joins a subtree
  $T_{u_i}$ with probability proportional to $|T_{u_i}|$ for
  $1 \le i \le j$, and $T_{u_i}$ gains exactly one vertex.
\end{proof}

Recall that a real-valued random variable $X$ is said to
\emph{stochastically dominate} a real-valued random variable $Y$ if,
for all $x \in \R$,
\[
  \Pr\{X \le x\} \le \Pr\{Y \le x\} .
\]

Let $F_{\alpha, \beta}$ be the cumulative distribution function of a
$\Beta(\alpha, \beta)$ random variable, and let
$f_{\alpha, \beta} : [0, 1] \to \R$ be its density. Recall that
\[
  f_{\alpha, \beta}(x) = \frac{1}{\mathrm{B}(\alpha, \beta)} x^{\alpha - 1} (1 - x)^{\beta - 1} .
\]
\begin{lem}\lemlabel{beta-dom-1}
  Let $0 < \alpha \le \beta \le \gamma$. Then,
  $\Beta(\beta, \gamma - \beta)$ stochastically dominates
  $\Beta(\alpha, \gamma - \alpha)$.
\end{lem}
\begin{proof}
  Let
  $(Y_1, Y_2, Y_3) \sim \Dirichlet(\alpha, \beta - \alpha, \gamma -
  \beta)$. Then,
  \begin{align*}
    Y_1 &\sim \Beta(\alpha, \gamma - \alpha) , \\
    Y_1 + Y_2 &\sim \Beta(\beta, \gamma - \beta) ,
  \end{align*}
  so, since the former is a partial sum of the latter,
  \[
    \Pr\{\Beta(\beta, \gamma - \beta) \le x\} = \Pr\{ Y_1 + Y_2 \le x\} \le \Pr\{ Y_1 \le x\} = \Pr\{\Beta(\alpha, \gamma - \alpha) \le x\} . \qedhere
  \]
\end{proof}

\begin{lem}\lemlabel{beta-dom-2}
  Let $0 < \alpha \le \beta$. Then, $\Beta(1,
  \alpha)$ stochastically dominates $\Beta(1, \beta)$.
\end{lem}
\begin{proof}
  We see, directly, for $x \in [0, 1]$,
  \[
    F_{1, \alpha}(x) = \alpha \int_0^x (1 - z)^{\alpha - 1} \dif z = 1 - (1 - x)^\alpha ,
  \]
  so clearly
  $F_{1, \alpha}(x) \le F_{1, \beta}(x) \iff (1 - x)^\alpha \ge (1 -
  x)^\beta \iff \alpha \le \beta$.
\end{proof}

\begin{lem}\lemlabel{beta-conc}
  Let $0 < \ell < k$. Then,
  \[
    \Pr\left\{\left|\Beta(k - \ell, \ell) - \frac{k - \ell}{k}\right| \le \frac{1}{\sqrt{k}} \right\} \ge \frac{3}{4} .
  \]
\end{lem}
\begin{proof}
  By Chebyshev's inequality,
  \[
    \Pr\left\{\left|\Beta(k - \ell, \ell) - \frac{k - \ell}{k}\right| \ge 2 \sqrt{\frac{(k - \ell) \ell}{k^2 (k + 1)}}\right\} \le \frac{1}{4} .
  \]
  By the arithmetic-geometric mean inequality,
  \[
    2 \sqrt{\frac{(k - \ell) \ell}{k^2 (k + 1)}} \le \frac{1}{\sqrt{k + 1}} ,
  \]
  and the result follows.
\end{proof}

\end{appendices}

\end{document}